\documentclass{amsart}
\usepackage{amssymb, amsfonts,amscd,verbatim,hyperref,cleveref,graphics}
\usepackage[latin1]{inputenc}

\usepackage{tikz-cd}
\usepackage{tikz}
\usetikzlibrary{matrix,calc,decorations.pathmorphing,shapes,arrows}
\usepackage{bbold}
\usepackage[all]{xy}

\newtheorem{thm}{Theorem}[section]
\newtheorem{cor}[thm]{Corollary}
\newtheorem{lem}[thm]{Lemma}
\newtheorem{prop}[thm]{Proposition}
\theoremstyle{definition}
\newtheorem{defn}[thm]{Definition}
\newtheorem{rem}[thm]{Remark}

\newtheorem{example}[thm]{Example}
\numberwithin{equation}{section}

\newcommand{\abs}[1]{\lvert#1\rvert}
\newcommand{\norm}[1]{\lVert#1\rVert}
\newcommand{\bigabs}[1]{\bigl\lvert#1\bigr\rvert}
\newcommand{\bignorm}[1]{\bigl\lVert#1\bigr\rVert}
\newcommand{\Bigabs}[1]{\Bigl\lvert#1\Bigr\rvert}
\newcommand{\biggabs}[1]{\biggl\lvert#1\biggr\rvert}
\newcommand{\Bignorm}[1]{\Bigl\lVert#1\Bigr\rVert}
\newcommand{\biggnorm}[1]{\biggl\lVert#1\biggr\rVert}

\renewcommand{\mid}{\::\:}
\DeclareMathOperator{\fbl}{FBL}

\newcommand{\weaksumnorm}[2]{\mu_{#1}(#2)}

\newcommand{\term}[1]{{\textit{\textbf{#1}}}}   % To introduce a term
\renewcommand{\leq}{\ensuremath{\leqslant}}
\renewcommand{\le}{\ensuremath{\leqslant}}

\renewcommand{\ge}{\ensuremath{\geqslant}}
\title[Free Banach lattices and convexity]{Free Banach lattices under convexity conditions}
\author[Jard\'on-S\'anchez]{H\'ector Jard\'on-S\'anchez}
\address{Department of Mathematics and Statistics, Fylde College, Lancaster University, Lancaster,
LA1 4YF, United Kingdom}
\email{hectorjardon@gmail.com}

\author[Laustsen]{Niels Jakob Laustsen}
\address{Department of Mathematics and Statistics, Fylde College, Lancaster University, Lancaster,
LA1 4YF, United Kingdom}
\email{n.laustsen@lancaster.ac.uk}

\author[Taylor]{Mitchell A.~Taylor}
\address{Department of Mathematics, University of California, Berkeley, CA, 94720, United
States}
\email{mitchelltaylor@berkeley.edu}

\author[Tradacete]{Pedro Tradacete}
\address{Instituto de Ciencias Matem\'aticas (CSIC-UAM-UC3M-UCM)\\
Consejo Superior de Investigaciones Cient\'ificas\\
C/ Nicol\'as Cabrera, 13--15, Campus de Cantoblanco UAM\\
28049 Madrid, Spain.}
\email{pedro.tradacete@icmat.es}

\author[Troitsky]{Vladimir G.~Troitsky}
\address{Department of Mathematical and Statistical Sciences, University of Alberta, Edmonton, Alberta T6G 2G1, Canada}
\email{troitsky@ualberta.ca}

\date{\today}

 %Banach lattices; vector lattices; free lattices; linear operators on ordered spaces


\begin{document}
\begin{abstract}
  We prove the existence of free objects in certain subcategories of
  Banach lattices, including $p$-convex Banach lattices, Banach
  lattices with upper $p$-estimates, and AM-spaces. From this we
  immediately deduce that projectively universal objects exist in each
  of these subcategories, extending results of Leung, Li, Oikhberg and
  Tursi (\emph{Israel J.\ Math.}~2019). In the $p$-con\-vex and
  AM-space cases, we are able to explicitly identify the norms of the
  free Banach lattices, and we conclude by investigating the structure
  of these norms in connection with nonlinear $p$-summing maps.
\end{abstract}

\maketitle

\section{Introduction}

Our objective is to construct free Banach lattices having certain
additional desirable properties, so let us begin by recalling the
fundamental definition: The \term{free Banach lattice over a Banach
  space~$E$} is a Banach lattice $\fbl[E]$ together with a linear
isometry $\phi_E\colon E\to \fbl[E]$ such that, for every Banach
lattice $X$ and every bounded linear operator $T\colon E\to X$, there
is a unique linear lattice homomorphism
$\widehat{T}\colon\fbl[E]\to X$ making the following diagram commute:

\begin{center}
 \tikzset{node distance=2cm, auto}
 \begin{tikzpicture}
  \node (C) {$\fbl[E]$};
  \node (P) [below of=C] {$E$};
  \node (Ai) [right of=P] {$X.$};
  \draw[->, dashed] (C) to node {$\exists ! \ \widehat{T}$} (Ai);
  \draw[<-] (C) to node [swap] {$\phi_E$} (P);
  \draw[->] (P) to node [swap] {$T$} (Ai);
 \end{tikzpicture}
\end{center}

De~Pag\-ter and Wick\-stead~\cite{dePW} initiated the focused study of
free Banach lattices by introducing the free Banach lattice generated
by a non\-empty set~$A$; in the above language, it corresponds to
$\fbl\bigl[\ell_1(A)\bigr]$. The construction of $\fbl[E]$ for an
arbitrary Banach space $E$ was carried out in~\cite{ART}, with further
research conducted in~\cite{APA}, \cite{ATV}, and~\cite{T19}. By now
it is firmly established that free Banach lattices provide a
fundamental tool for understanding the interplay between Banach-space
and Banach-lattice properties. In particular, spaces of the
form~$\fbl[\ell_2(A)]$ for an uncountable set~$A$ are used in
\cite[Section~5]{ART} to resolve an open problem of Die\-stel.

However, being universal, free Banach lattices usually lack classical
properties such as reflexivity and $p$-convexity. To counteract this,
we will restrict the target spaces in the above diagram to only those
Banach lattices~$X$ which satisfy some fixed property~$P$, and then
look to replace $\fbl[E]$ with a Banach lattice satisfying~$P$. More
specifically, we shall prove the following result.

\begin{thm}\label{p-convex exists}
  Let $E$ be a Banach space and $1\le p\le\infty$. There exists a pair
  $\bigl(\fbl^{(p)}[E], \phi_E\bigr),$ where $\fbl^{(p)}[E]$ is a
  $p$-convex Banach lattice with $p$-convexity constant~$1$ and
  $\phi_E \colon E\to\fbl^{(p)}[E]$ is a linear isometry, with the
  following universal property: For every $p$-convex Banach lattice
  $X$ and every bounded linear operator $T\colon E\to X,$ there exists
  a unique linear lattice homomorphism
  $\widehat{T}\colon\fbl^{(p)}[E]\to X$ such that
  $\widehat{T}\circ\phi_E=T$.  Moreover,
  $\bignorm{\widehat{T}}\le M^{(p)}(X)\,\norm{T},$ where $M^{(p)}(X)$
  denotes the $p$-con\-vexity constant of $X,$ and the pair
  $\bigl(\fbl^{(p)}[E],\phi_E\bigr)$ is essentially unique.
\end{thm}

In later sections we will give an explicit description of the Banach
lattices whose existence is asserted in \Cref{p-convex exists} and
show analogous results when $p$-convexity is replaced with upper
$p$-estimates or being a (unital) AM-space. Of course, one cannot
expect a version of \Cref{p-convex exists} for reflexivity or
$p$-concavity without any restrictions on~$E$, as not all Banach
spaces embed into such Banach lattices.

An outline of the paper is as follows: \Cref{section:Prelim} contains
some preliminary material, primarily concerning function calculus,
that we require in \Cref{section:largestconvexnorm} when showing that
$\fbl^{(p)}[E]$ exists.  In fact, the main result of
\Cref{section:largestconvexnorm} is somewhat more general than
\Cref{p-convex exists}, as it is stated in terms of a new notion which
we call ``$\mathcal D$-convexity'' and which encompasses both
$p$-convexity and upper $p$-estimates.  The approach taken in
\Cref{section:largestconvexnorm} is similar to that of the recent
paper~\cite{T19}, in which~$\fbl[E]$ is constructed as the completion
of the free vector lattice $\operatorname{FVL}[E]$
(see~\cite{Baker,Bleier}) under a certain ``maximal'' lattice
norm. However, some additional work is needed to make sense of
function calculus.

An advantage of this abstract approach is that it allows us to
construct free objects in various other subcategories of Banach
lattices, while a significant drawback is that it does not provide any
concrete description of these spaces, notably leaving it open whether
$\fbl^{(p)}[E]$ can be realized as a vector lattice of
functions. Despite this, the universal properties of these spaces are
powerful enough to establish several results, which we do in
\Cref{Basic}. \Cref{AMstuff} is devoted to free AM-spaces and free
$C(K)$-spaces.

Then, in Section \ref{section:description} we return to the beginnings
by giving an alternative, explicit description of $\fbl^{(p)}[E]$ as a
sublattice of the vector lattice of real-valued functions defined on
the dual Banach space~$E^*$ of~$E$, thus in particular resolving the
above problem. This section can be read independently of the previous
ones.

In order to motivate the candidate norm, let us recall the
construction of the space $\fbl[E]$ and its norm
from~\cite[Section~2]{ART}: For any function
$f\colon E^*\to\mathbb R$, define
\begin{displaymath}
  \norm{f}_{\fbl[E]}=
  \sup\Bigl\{\sum_{k=1}^n \bigabs{f(x_k^*)}\mid n\in\mathbb N, \,
   x_1^*,\dots,x_n^*\in E^*, \, \sup_{x\in B_E}\sum_{k=1}^n
  \bigabs{x_k^\ast(x)} \le 1\Bigr\}.
\end{displaymath}
The set $H_1[E]$ of positively homogeneous functions
$f\colon E^*\to\mathbb R$ with $\norm{f}_{\fbl[E]}<\infty$ turns out
to be a Banach lattice with respect to the pointwise operations and
this norm, and~$\fbl[E]$ is defined as the closure in~$H_1[E]$ of the
sublattice generated by the set $\{\delta_x \mid x\in E\}$, where
$\delta_x\colon E^*\to\mathbb R$ is the evaluation map given by
$\delta_x(x^*)=x^*(x)$, together with the linear isometry
$\phi_E\colon E\to\fbl[E]$ defined by $\phi_E(x)=\delta_x$.

We shall show that analogously, for $1<p<\infty$, $\fbl^{(p)}[E]$ can
equivalently be defined as the closure of the sublattice generated by
$\{\delta_x \mid x\in E\}$ in the Banach lattice of positively
homogeneous functions $f\colon E^*\to\mathbb R$ for which the quantity
\begin{equation}\label{Eq:FBLpNorm}
	\norm{f}_p =
	\sup\biggl\{\Bigl(\sum_{k=1}^n\bigabs{f(x_k^*)}^p\Bigr)^{\frac{1}{p}} \mid
        n\in\mathbb N,\, x_1^*,\dots,x_n^*\in E^*,\,
        \sup_{x\in B_E}\sum_{k=1}^n\bigabs{x_k^*(x)}^p\le 1\biggr\}
\end{equation}
is finite.  This implies in particular that there is a continuous
linear injection of~$\fbl^{(p)}[E]$ into the space~$C(B_{E^*})$ of
continuous functions on the closed unit ball~$B_{E^*}$ of~$E^*$,
equipped with the relative weak$^*$ topology. Although the
form~\eqref{Eq:FBLpNorm} of $\norm{\cdot}_p$ is clearly motivated
by~\cite{ART}, the proof that it is indeed the free $p$-convex norm
requires quite different techniques. Having this explicit expression
will be particularly useful in certain computations.

For the reader who is familiar with the theory of $p$-summing operators, the above expression for the free $p$-convex norm has another interpretation: A function $f\colon E^*\to \mathbb R$ for which $\norm{f}_p<\infty$ maps weakly $p$-summable sequences in $E^*$ to (strongly) $p$-summable sequences in~$\mathbb R$. In Section \ref{section:pqsumming} we devote our attention to the spaces of positively homogeneous functions from $E^*$ to $\mathbb R$ with finite $(p,q)$-summing norm, and explore classical arguments such as the Dvorezky--Rogers Theorem and Pietsch's Domination Theorem in this nonlinear setting.

\section{Preliminaries}\label{section:Prelim}
\noindent
Our notation and terminology are mostly standard and will be
introduced as and when needed. A few general conventions are as
follows. All vector spaces, including vector lattices, Banach spaces
and Banach lattices, are real.  The terms ``operator'' and ``lattice
homomorphism'' will be synonymous with ``bounded linear operator'' and
``linear lattice homomorphism'', respectively.  A ``sublattice'' of a
vector lattice will mean a linear subspace which is closed under
finite suprema and infima. We shall repeatedly use the elementary fact
that the sublattice generated by a subset~$W$ of a vector lattice is
given by
\begin{equation}\label{eq:ABsublattice}
  \biggl\{ \bigvee_{j=1}^n x_j - \bigvee_{j=1}^n y_j \mid  n\in\mathbb{N},\,
  x_1,\dots,x_n,y_1,\dots,y_n\in \operatorname{span}W\biggr\};
\end{equation}
see, e.g.,  \cite[p.~204, Exer\-cise~8(b)]{AB}.

For a positive element~$e$ of a vector lattice~$X$, $I_e$ denotes the
order ideal of~$X$ generated by~$e$, that is,
\begin{equation*}
  I_e = \bigl\{ x\in X \mid  \abs{x}\le\lambda
  e\ \text{for some}\ \lambda\in[0,\infty)\bigr\}.
\end{equation*}
We can endow~$I_e$ with the lattice seminorm defined by
\begin{equation}\label{eq:IeNorm}
  \norm{x}_e = \inf\bigl\{
  \lambda\in[0,\infty) \mid  \abs{x}\le\lambda e\bigr\}
\end{equation} 
for every $x\in I_e$. This seminorm  is a norm if~$X$ is Ar\-chi\-me\-dean.
\subsection*{Function calculus}
The definitions of $p$-convexity and upper $p$-estimates can be stated
by means of function calculus, which is a standard tool in Banach
lattices (see, e.g., \cite[1.d]{LT2}). However, since our
constructions will require us to work in more general vector lattices,
we bring some basic facts to the reader's attention. We essentially
need only what is contained in~\cite{LT}.

For $m\in\mathbb N$, $\mathcal{H}_m$ denotes the vector lattice of
continuous, positively homogeneous, real-valued functions
on~$\mathbb{R}^m$. Clearly it contains the $k^{\text{th}}$ coordinate
projection $\pi_k\colon (t_1,\dots,t_m)\mapsto t_k$ for each
$k\in\{1,\dots,m\}$. We say that a vector lattice~$X$ \term{admits a
  positively ho\-mo\-geneous continuous function calculus} if, for
every $m\in \mathbb{N}$ and every $m$-tuple
$\boldsymbol{x}=(x_1,\dots,x_m)\in X^m$, there is a lattice
homomorphism $\Phi_{\boldsymbol{x}}\colon\mathcal{H}_m\to X$ such that
\begin{equation}\label{defnphcfcEq}
  \Phi_{\boldsymbol{x}}(\pi_k) = x_k
\end{equation}
for each $k\in\{1,\ldots,m\}$. In this case, we refer to the map
$\boldsymbol{x}\mapsto\Phi_{\boldsymbol{x}}$ (or
simply~$\Phi_{\boldsymbol{x}}$) as a \term{positively homogeneous
  continuous function calculus} for~$X$. In line with common practice,
for $h\in\mathcal{H}_m$, we usually use the shorter and more
suggestive notation $h(x_1,\dots,x_m)$ instead
of~$\Phi_{\boldsymbol{x}}(h)$.
   
It is well known that every uniformly complete vector lattice (in
particular, every Banach lattice) admits a positively homogeneous
continuous function calculus.  The following more precise
characterization was given in \cite[Theorem~1.3]{LT}: An
Ar\-chi\-medean vector lattice $X$ admits a positively homogeneous
continuous function cal\-cu\-lus if and only if it is \term{finitely
  uniformly complete} in the sense that, for every $m\in\mathbb N$ and
$x_1,\ldots,x_m\in X$, there is a positive element $e\in X$ such that
$e\ge\bigvee_{j=1}^m\abs{x_j}$ and the norm $\norm{\,\cdot\,}_e$ given
by~\eqref{eq:IeNorm} is complete on the closed sub\-lattice of
$\bigl(I_e,\norm{\,\cdot\,}_e\bigr)$ generated by $x_1,\ldots,x_m$. We shall
freely use this result in the following with\-out any further
reference.

We can define a norm $\norm{\,\cdot\,}_{\mathcal{H}_m}$
on~$\mathcal{H}_m$ by identifying it with $C(S_{\ell_\infty^m})$ via
the restriction map $h\mapsto h|_{S_{\ell_\infty^m}}$,
where~$S_{\ell_\infty^m}$ denotes the unit sphere
of~$\ell_\infty^m$. The sublattice of~$\mathcal{H}_m$ generated by
$\{\pi_k \mid 1\le k\le m\}$ is dense with respect to this norm.

We shall now establish some basic facts that we require in later
sections. They may be known, but as we have been unable to find any
precise references in the literature, we include their proofs.  We
begin with a result which will imply that when an Archimedean vector
lattice admits a positively homogeneous continuous function calculus,
it is unique in a very strong sense.

\begin{lem}\label{NJLlemma2Aug}
  Let $T\colon X\to Y$ be a lattice homomorphism between two
  Archi\-medean vector lattices~$X$ and~$Y,$ let $x_1,\ldots,x_m\in X$
  for some $m\in\mathbb N,$ and define $e = \bigvee_{j=1}^m\abs{x_j}$
  in $X_+$. Then:
\begin{enumerate}
\item\label{NJLlemma2AugI} $T$ maps the ideal~$I_e$ of~$X$ into the
  ideal~$I_{Te}$ of~$Y,$ and $\norm{Tx}_{Te}\le\norm{x}_e$ for every
  $x\in I_e$.
\item\label{NJLlemma2AugII} $\norm{Tx}_{Te} = \norm{x}_e$ for each
  $x\in I_e$ if and only if the restriction of~$T$ to~$I_e$ is
  injective.
\end{enumerate}
Suppose that $\Phi_{\boldsymbol{x}}\colon \mathcal{H}_m\to X$ is a
lattice homomorphism which satisfies~\eqref{defnphcfcEq}. Then:
\begin{enumerate}
\setcounter{enumi}{2}
\item\label{NJLlemma2AugIII} The image of~$\Phi_{\boldsymbol{x}}$ is
  contained in~$I_e$, and $\Phi_{\boldsymbol{x}}$ is bounded with norm
  at most~$1$ when considered an operator
  into~$\bigl(I_e,\norm{\,\cdot\,}_e\bigr)$.
\item\label{NJLlemma2AugIIII} The composite map
  $T\circ\Phi_{\boldsymbol{x}}$ is the unique lattice homomorphism
  from~$\mathcal{H}_m$ into~$Y$ which maps $\pi_k$ to $Tx_k$ for each
  $k=1,\ldots,m$.
  \end{enumerate}
\end{lem}

\begin{proof}
  \eqref{NJLlemma2AugI}. For every $x\in I_e$, we can find
  $\lambda\in[0,\infty)$ such that $\abs{x}\le\lambda e$.  Since~$T$
  is a lattice homomorphism, we have $\abs{Tx}\le\lambda Te$, so
  $Tx\in I_{Te}$ with $\norm{Tx}_{Te}\le\lambda$. Now the conclusion
  follows by taking the infimum over all~$\lambda$ with
  $\abs{x}\le\lambda e$.

  \eqref{NJLlemma2AugII}. The forward implication is clear
  because~$\norm{\,\cdot\,}_e$ is a norm on~$I_e$.

  Conversely, suppose that the restriction~$T|_{I_e}$ is injective, and
  consider $x\in I_e$ and $\lambda\in[0,\infty)$ with 
  $\abs{Tx}\le\lambda Te$. Then $T\bigl(\lambda e - \abs{x}\bigr) =
  \lambda Te -\abs{Tx}\ge 0$, so that
  \begin{displaymath}
    0 = 0\wedge T\bigl(\lambda e - \abs{x}\bigr)
    = T\Bigl(0\wedge\bigl(\lambda e - \abs{x}\bigr)\Bigr).
  \end{displaymath}
  This implies that $0\wedge\bigl(\lambda e-\abs{x}\bigr) =0$ by the
  injectivity of~$T|_{I_e}$, that is, $\lambda e\ge\abs{x}$. Hence
  $\lambda\ge\norm{x}_e$, and taking the infimum over all~$\lambda$
  with $\abs{Tx}\le \lambda Te$, we conclude that
  $\norm{Tx}_{Te}\ge\norm{x}_e$. The opposite inequality was shown
  in~\eqref{NJLlemma2AugI}.

  \eqref{NJLlemma2AugIII}.  This is a special case
  of~\eqref{NJLlemma2AugI}, applied with $T=\Phi_{\boldsymbol{x}}$ and
  $x_k =\pi_k|_{S_{\ell_\infty^m}}$ for each $k=1,\ldots,m$. To
  see this, recall our identification of~$\mathcal{H}_m$ with
  $C(S_{\ell_\infty^m})$, and observe that
  $\bigvee_{k=1}^m\bigl\lvert\pi_k|_{S_{\ell_\infty^m}}\bigr\rvert =
  \mathbb{1}$ is a strong unit in~$C(S_{\ell_\infty^m})$, with the
  corresponding lattice norm~\eqref{eq:IeNorm} being equal to the
  uniform norm~$\norm{\,\cdot\,}_\infty$.

  \eqref{NJLlemma2AugIIII}. Only the uniqueness statement is not
  clear. To prove it, let $S\colon\mathcal{H}_m\to Y$ be any
  lattice homo\-mor\-phism with $S(\pi_k)=Tx_k$ for each
  $k=1,\ldots,m$.  By~\eqref{NJLlemma2AugI}
  and~\eqref{NJLlemma2AugIII}, we may regard $T\colon I_e\to I_{Te}$,
  $\Phi_{\boldsymbol{x}}\colon\mathcal{H}_m\to I_e$, and
  $S\colon\mathcal{H}_m\to I_{Te}$ as bounded lattice homomorphisms
  with respect to the specified domains and codomains, where $I_e$
  and~$I_{Te}$ are given the norms~$\norm{\,\cdot\,}_e$
  and~$\norm{\,\cdot\,}_{Te}$, respectively. Then
  $T\circ\Phi_{\boldsymbol{x}}\colon \mathcal{H}_m\to I_{Te}$ is also
  bounded, and therefore it is equal to~$S$ because
  $(T\circ\Phi_{\boldsymbol{x}})(\pi_k) = S(\pi_k)$ for each
  $k=1,\ldots,m$ and $\{\pi_k \mid 1\le k\le m\}$
  generates a dense sublattice of~$\mathcal{H}_m$.
\end{proof}  

\begin{cor}\label{CorSublatticephcfc} 
  Let~$X$ be a finitely uniformly complete Archimedean vector lattice,
  and let~$Y$ be a sublattice of~$X$.  Then~$Y$ is finitely uniformly
  complete if and only if
  $\Phi_{\boldsymbol{y}}(\mathcal{H}_m)\subseteq Y$ for every
  $m\in\mathbb N$ and $\boldsymbol{y}\in Y^m$,
  where~$\Phi_{\boldsymbol{y}}$ denotes the positively homogeneous
  continuous function calculus for~$X$.
\end{cor}

\begin{proof}
  The implication $\Leftarrow$ is clear, while the converse follows
  from the uniqueness statement in
  \Cref{NJLlemma2Aug}\eqref{NJLlemma2AugIIII}, applied in the case
  where $T\colon Y\to X$ is the inclusion map.
\end{proof}  

\begin{cor}\label{CorFunctionCalcCommutesWithLatHom}
  Let $T\colon X\to Y$ be a lattice homomorphism between two finitely
  uniformly complete Archi\-medean vector lattices~$X$ and~$Y$.  Then
  $T\bigl(h(x_1,\dots,x_m)\bigr)=h(Tx_1,\dots,Tx_m)$ for every $m\in\mathbb N,$
  $h\in\mathcal{H}_m,$ and $x_1,\ldots,x_m\in X$.
\end{cor}

\begin{proof}
  This is immediate from  \Cref{NJLlemma2Aug}\eqref{NJLlemma2AugIIII}.
\end{proof}  

Our next result involves the following standard notion: A
sequence~$(x_k)$ in an Ar\-chi\-medean vector lattice $X$
\term{converges uniformly} to $x\in X$ if $X$ contains a positive
element~$e$ such that, for every $\varepsilon\in(0,\infty)$, there is
$k_0\in\mathbb N$ with $\abs{x_k-x}\le\varepsilon e$ whenever
$k\ge k_0$. If explicit reference to the element~$e$ is required, we
call it a \term{regulator} and say that it \term{witnesses} the
convergence.

\begin{lem}\label{lemma:unifconv}
  Let $X$ be a finitely uniformly complete Archimedean vector lattice,
  let $m\in\mathbb N,$ and suppose that $(x_k^1),\ldots, (x_k^m)$ are
  sequences in~$X$ which converge uni\-form\-ly to $x^1,\ldots,x^m,$
  respectively.  Then the sequence $\bigl(h(x^1_k,\ldots,x^m_k)\bigr)$ converges
  uni\-form\-ly to~$h(x^1,\ldots,x^m)$ for every $h\in \mathcal{H}_m$.
\end{lem}

\begin{proof}
  For each $j=1,\ldots,m$, let $e^j\in X_+$ be a regulator which
  witnesses that $(x^j_k)$ converges uniformly to~$x^j$.  We shall
  show that
  \begin{equation}\label{lemma:unifconvEq0}
    f = \bigvee_{j=1}^m\abs{x^j}+ 
    \sum_{j=1}^me^j+ \bigvee_{j=1}^m(\abs{x^j} + e^j)\in X_+
  \end{equation}
  is a regulator witnessing that $\bigl(h(x^1_k,\ldots,x^m_k)\bigr)$ converges
  uni\-form\-ly to~$h(x^1,\ldots,x^m)$ for any $h\in \mathcal{H}_m$.
  To this end, let $\varepsilon>0$.  Since the sublattice generated by
  $\{\pi_1,\ldots, \pi_m\}$ is dense in~$\mathcal{H}_m$, it contains a
  function~$\ell$ such that
  $\norm{h-\ell}_{\mathcal{H}_m}\le\varepsilon$, that is,
  \begin{equation}\label{lemma:unifconvEq1}
    \Bigabs{h(t_1,\ldots,t_m)-\ell(t_1,\ldots,t_m)}\le
    \varepsilon\bigvee_{j=1}^m\abs{t_j}
  \end{equation}
  for every $t_1,\ldots,t_m\in\mathbb R$. By~\eqref{eq:ABsublattice},
  we can express~$\ell$ as
  \begin{displaymath}
    \ell=\bigvee_{i=1}^n\Bigl(\sum_{j=1}^m\alpha_{ij}\pi_j\Bigr)
    -\bigvee_{i=1}^n\Bigl(\sum_{j=1}^m\beta_{ij}\pi_j\Bigr)
  \end{displaymath}
  for some $n\in \mathbb{N}$ and coefficients
  $\alpha_{ij},\beta_{ij}\in\mathbb R$. It follows that for every
  $k\in\mathbb N$,
 \begin{align}\label{lemma:unifconvEq2}
    \Bigabs{\ell&(x^1,\dots,x^m)-\ell(x_k^1,\dots,x_k^m)}\\
    &\le
    \biggabs{\bigvee_{i=1}^n\Bigl(\sum_{j=1}^m\alpha_{ij}x^j\Bigr) -
    \bigvee_{i=1}^n\Bigl(\sum_{j=1}^m\alpha_{ij}x_k^j\Bigr)} +
    \biggabs{\bigvee_{i=1}^n\Bigl(\sum_{j=1}^m\beta_{ij}x^j\Bigr) -
    \bigvee_{i=1}^n\Bigl(\sum_{j=1}^m\beta_{ij}x_k^j\Bigr)}\notag\\
    &\le\bigvee_{i=1}^n\sum_{j=1}^m\abs{\alpha_{ij}}\ \abs{x^j - x_k^j}
      +\bigvee_{i=1}^n\sum_{j=1}^m\abs{\beta_{ij}}\ \abs{x^j-x_k^j}
      \le C\sum_{j=1}^m\abs{x^j - x_k^j}\notag,
  \end{align}
  where
  $C=\bigvee_{i,j}\abs{\alpha_{ij}}+\bigvee_{i,j}\abs{\beta_{ij}}$. Choose
  $k_0\in\mathbb N$ such that
  $\abs{x^j - x_k^j}\le\frac{\varepsilon}{\varepsilon\,\vee\,C}e^j$ for every
  $k\ge k_0$ and $j=1,\ldots,m$. Then, combining the
  estimate
\begin{multline*}
   \Bigabs{h(x^1,\ldots,x^m)-h(x_k^1,\ldots,x_k^m)}
   \le
   \Bigabs{h(x^1,\ldots,x^m)-\ell(x^1,\ldots,x^m)}\\ +
   \Bigabs{\ell(x^1,\ldots,x^m)-\ell(x_k^1,\ldots,x_k^m)}
   + \Bigabs{\ell(x_k^1,\ldots,x_k^m)- h(x_k^1,\ldots,x_k^m)}
\end{multline*}
with~\eqref{lemma:unifconvEq1}--\eqref{lemma:unifconvEq2}, we obtain
\begin{displaymath}
  \Bigabs{h(x^1,\ldots,x^m)-h(x_k^1,\ldots,x_k^m)}
  \le\varepsilon\bigvee_{j=1}^m\abs{x^j}+ 
  \varepsilon\sum_{j=1}^me^j+ \varepsilon\bigvee_{j=1}^m\abs{x^j_k}
  \le \varepsilon f
 \end{displaymath}
for every  $k\geqslant k_0$, where~$f$ is defined by~\eqref{lemma:unifconvEq0}. This completes the proof.
\end{proof}

\begin{cor}\label{Continuitylemma}
  Let~$X$ be a Banach lattice, and let $h\in\mathcal{H}_m$ for
  some $m\in\mathbb N$. Then the map from $X^m$ to $X$ defined via
  $(x_1,\ldots,x_m)\mapsto h(x_1,\ldots,x_m)$ is
  continuous with respect to the norm
  $\bignorm{(x_1,\ldots,x_m)}=\bigvee_{j=1}^m\norm{x_j}$ on~$X^m$.
\end{cor}

\begin{proof}
  Assume the contrary. Then there are sequences
  $(x_k^1),\ldots, (x_k^m)$ in~$X$ which norm converge to
  $x^1,\ldots,x^m,$ respectively, but the sequence
  $\bigl(h(x^1_k,\ldots,x^m_k)\bigr)$ is not norm convergent
  to~$h(x^1,\ldots,x^m)$. Passing to a subsequence, we may suppose
  that there is $\varepsilon>0$ such that
  \begin{equation}\label{ContinuitylemmaEq1}
  \Bignorm{h(x^1_k,\ldots,x^m_k) - h(x^1,\ldots,x^m)}\ge\varepsilon
  \end{equation}
  for every $k\in\mathbb N$.  Since every norm convergent sequence has
  a uniformly convergent subsequence, we may further suppose that
  $(x_k^1),\ldots, (x_k^m)$ converge uniformly to $x^1,\ldots,x^m,$
  respectively. Then \Cref{lemma:unifconv} implies that
  $\bigl(h(x^1_k,\ldots,x^m_k)\bigr)$ converges uniformly
  to~$h(x^1,\ldots,x^m)$, which
  contradicts~\eqref{ContinuitylemmaEq1}.
  \end{proof}

\section{Construction of free $\mathcal{D}$-convex Banach lattices}\label{section:largestconvexnorm}

A Banach lattice $X$ is \term{$p$-convex} for some $p\in[1,\infty]$ if there exists a constant $M\ge 1$ such that for every $m\in\mathbb N$ and $x_1,\dots,x_m\in X$, we have 
\begin{equation}\label{Defn of convex}
  \biggnorm{\Bigl(\sum_{k=1}^m\abs{x_k}^p\Bigr)^{\frac{1}{p}}}\le
  M\Bigl(\sum_{k=1}^m\norm{x_k}^p\Bigr)^{\frac{1}{p}}.
\end{equation}
Clearly it suffices to verify this inequality in the case where
$x_1,\dots,x_m$ are positive. The least constant~$M$ for
which~\eqref{Defn of convex} is valid is denoted $M^{(p)}(X)$ and is
called the \term{$p$\nobreakdash-con\-vex\-ity constant} of
$X$. If~\eqref{Defn of convex} is assumed to hold only for pairwise
disjoint elements $x_1,\dots,x_m\in X$, then $X$ is said to satisfy an
\term{upper $p$\nobreakdash-es\-ti\-mate}. The reader is referred to
\cite{LT2} for further information on such Banach lattices.

In this section we present a general construction which yields the
existence of both free $p$-convex Banach lattices and free Banach
lattices with upper $p$-estimates. For this, we need a definition
encompassing both concepts.

Let $m\in\mathbb N$, and recall that~$\mathcal{H}_m$ denotes the set
of all continuous, positively homogeneous functions from
$\mathbb{R}^m$ to~$\mathbb{R}$. We say that a function
$h\in \mathcal{H}_m$ is \term{monotone on $\mathbb{R}^m_+$} if
$h(t_1,\dots,t_m)\le h(s_1,\dots,s_m)$ whenever
$t_1,\dots,t_m,s_1,\dots,s_m\in \mathbb{R}$ with $0\le t_k\le s_k$
for each $k=1,\ldots,m$. This, of course, implies that
$h(t_1,\dots,t_m)\ge 0$ whenever $t_1,\dots,t_m\ge 0$. We denote by
$\mathcal{H}_m^{>0}$ the set of all continuous, positively homogeneous
functions which are monotone on~$\mathbb{R}^m_+$.

By a \term{convexity condition}, we understand a triple
$\mathcal{D}=(\mathcal{G}, M,\vartheta)$, where $\mathcal{G}$ is a
non\-empty subset of~$\bigcup_{m=1}^\infty \mathcal{H}_m^{>0}$ and
$M\colon\mathcal{G}\to [1,\infty)$ and
$\vartheta\colon\mathcal{G}\to \{0,1\}$ are any functions. Given a
convexity condition $\mathcal{D}=(\mathcal{G}, M,\vartheta)$, we set
$\mathcal{G}_m = \mathcal{G}\cap\mathcal{H}_m^{>0}$ for
$m\in\mathbb N$ and say that a lattice seminorm $\nu$ on a finitely
uniformly complete Ar\-chi\-me\-dean vector lattice~$X$ is
\term{$\mathcal{D}$-convex} if, for every $m\in\mathbb N$ and
$g\in \mathcal{G}_m$, the in\-equality
\begin{equation}\label{Defn:Dconvex}
   \nu\bigl(g(x_1,\dots,x_m)\bigr)\le M(g)\cdot g\bigl(\nu(x_1),\dots, \nu(x_m)\bigr)
\end{equation}
holds for all pairwise disjoint elements
$x_1,\dots, x_m\in X_+$ if $\vartheta(g)=0$, and for all
elements $x_1,\dots, x_m\in X_+$ if $\vartheta(g)=1$.  A Banach
lattice is \term{$\mathcal{D}$-convex} if its norm is
$\mathcal{D}$-convex.

Note that every closed sublattice of a $\mathcal{D}$-convex Banach
lattice is $\mathcal{D}$-convex by uniqueness of the function
calculus.

\begin{rem}\label{RemarkMainCasesofDconvex}
  It should be clear that $p$-convexity and upper $p$-estimates can be
  re\-covered easily from $\mathcal{D}$-convexity. Indeed, for
  $p\in[1,\infty]$, let
  $\mathcal{G}^p=\bigl\{\norm{\,\cdot\,}_{\ell_p^m} \mid m\in
  \mathbb{N}\bigr\}$ be the collection of all
  $\ell_p^m$\nobreakdash-norms, and
  let~$M\colon\mathcal{G}^p\to[1,\infty)$ be a constant function.
  Then, choosing $\vartheta$ to be the con\-stant function~$0$, we
  obtain an upper $p$-estimate, while choosing $\vartheta$ to be the
  con\-stant function~$1$ gives $p$\nobreakdash-con\-vexi\-ty (with
  constant $M^{(p)}\le M$).
\end{rem}

\begin{lem}\label{lemma:DconvexCompletion}
  Let~$\mathcal{D}$ be a convexity condition.  The completion of a
  finitely uniformly complete Archimedean vector lattice with respect
  to a $\mathcal{D}$-convex lattice norm is a
  $\mathcal{D}$\nobreakdash-con\-vex Banach lattice.
\end{lem}

\begin{proof} 
  Suppose that $\mathcal{D}=(\mathcal{G},M,\vartheta)$, and let~$X$ be
  a finitely uniformly complete Ar\-chi\-me\-dean vector lattice
  endowed with a $\mathcal{D}$-convex lattice
  norm~$\norm{\,\cdot\,}$.  Its completion~$\widetilde{X}$ is a
  Banach lattice, so we just need to verify that~$\widetilde{X}$ is
  $\mathcal{D}$\nobreakdash-con\-vex.  Suppose that
  $g\in\mathcal{G}_m$ for some $m\in\mathbb N$.
 
  We begin with the case $\vartheta(g)=1$ as it is easier. Given
  $x^1,\ldots,x^m\in\widetilde{X}_+$, choose sequences
  $(x_k^1),\ldots, (x_k^m)$ in~$X_+$ which converge in norm
  to~$x^1,\ldots,x^m$, respectively. The $\mathcal{D}$-convexity of
  the norm on~$X$ coupled with the continuity of~$g$ implies that
  \begin{displaymath}
    \bignorm{g(x_k^1,\ldots,x_k^m)}\le
    M(g)\cdot g\bigl(\norm{x_k^1},\ldots,\norm{x_k^m}\bigr)
    \xrightarrow{k\to \infty}
    M(g)\cdot g\bigl(\norm{x^1},\ldots,\norm{x^m}\bigr).
  \end{displaymath}
  Now~\eqref{Defn:Dconvex} follows because the left-hand side
  converges to $\bignorm{g(x^1,\ldots,x^m)}$ by
  \Cref{Continuitylemma}.

  Suppose instead that $\vartheta(g)=0$, and let
  $x^1,\ldots,x^m\in\widetilde{X}_+$ be pairwise disjoint. As before,
  choose sequences $(y_k^1),\ldots, (y_k^m)$ in~$X_+$ which converge
  in norm to~$x^1,\ldots,x^m$, respectively.  We can ``disjointify''
  these sequences by defining
  \begin{displaymath}
    x_k^i=\bigwedge\bigl\{y_k^i-y_k^i\wedge y^j_k \mid  1\le j\le m,\, j\ne i\bigr\}    
  \end{displaymath}
  for every $k\in\mathbb N$ and $i=1,\ldots,m$.  Then
  $x_k^1,\ldots,x_k^m$ are pairwise disjoint for every
  $k\in\mathbb N$, and the sequence $(x_k^i)$ converges in norm
  to~$x^i$ for each $i=1,\ldots,m$ because the lattice
  operations are continuous and $x^1,\ldots,x^m$ are pairwise
  disjoint.  We can now complete the proof as in the case
  $\vartheta(g)=1$.
\end{proof}
  
\begin{thm}\label{D-convex exists}
  Let $E$ be a Banach space and~$\mathcal{D}$ a convexity condition,
  as defined above. There exists a pair
  $\bigl(\fbl^{\mathcal{D}}[E],\phi_E^\mathcal{D}\bigr),$ where
  $\fbl^{\mathcal{D}}[E]$ is a $\mathcal{D}$-convex Banach lattice and
  $\phi_E^\mathcal{D}\colon E\to\fbl^{\mathcal{D}}[E]$ a linear
  isometry, with the following universal property: For every
  $\mathcal{D}$-convex Banach lattice $X$ and every operator
  $T\colon E \to X,$ there exists a unique lattice homomorphism
  $\widehat{T}\colon \fbl^{\mathcal{D}}[E] \to X$ such that
  $\widehat{T}\circ \phi_E^\mathcal{D}=T,$ i.e., the following diagram
  commutes:
  \begin{center}
    \tikzset{node distance=2cm, auto}
    \begin{tikzpicture}
      \node (C) {$\fbl^{\mathcal{D}}[E]$};
      \node (P) [below of=C] {$E$};
      \node (Ai) [right of=P] {$X.$};
      \draw[->, dashed] (C) to node {$\widehat{T}$} (Ai);
      \draw[<-] (C) to node [swap] {$\phi_E^\mathcal{D}$} (P);
      \draw[->] (P) to node [swap] {$T$} (Ai);
    \end{tikzpicture}
  \end{center}
  Moreover, $\norm{\widehat{T}}=\norm{T}$.
\end{thm}

\begin{proof}
  Throughout this proof, we work in the Archimedean vector
  lattice~$\mathbb R^{B_{E^*}}$ of real-valued functions defined on
  the closed unit ball~$B_{E^*}$ of the dual space~$E^*$. Being a
  vector lattice of functions, $\mathbb R^{B_{E^*}}$~admits a
  positively homogeneous continuous function calculus which is defined
  pointwise, that is,
  \begin{equation}\label{DconvexExistsEq1}
    \bigl(h(f_1,\ldots,f_m)\bigr)(x^*) = h\bigl(f_1(x^*),\ldots,f_m(x^*)\bigr)    
  \end{equation}
  for $m\in\mathbb N$, $h\in\mathcal{H}_m$,
  $f_1,\ldots,f_m\in\mathbb R^{B_{E^*}}$, and $x^*\in B_{E^*}$.

  Regarding the point evaluations $\delta_x\colon x^*\mapsto x^*(x)$
  for $x\in E$ as elements of~$\mathbb R^{B_{E^*}}$, we can define
  \begin{equation}\label{defn:Y0}
    Y_0=\Bigl\{h(\delta_{x_1},\ldots,\delta_{x_m})\mid  m\in\mathbb N,\,
    h\in\mathcal{H}_m,\, x_1,\ldots,x_m\in E\Bigr\}\subseteq\mathbb R^{B_{E^*}}.
  \end{equation}
  The set $\bigcup_{m=1}^\infty\mathcal{H}_m$ is closed under
  compositions in the sense that it contains the function
  \begin{displaymath}
    (t_1^1,\ldots,t_{m_1}^1,\ldots,t_1^n,\ldots,t_{m_n}^n)\mapsto
    h\bigl(g_1(t_1^1,\ldots,t_{m_1}^1),\ldots,g_n(t_1^n,\ldots,t_{m_n}^n)\bigr) 
  \end{displaymath}
  whenever $n,m_1,\ldots,m_n\in\mathbb N$, $h\in\mathcal{H}_n$, and
  $g_1\in\mathcal{H}_{m_1},\ldots,g_n\in\mathcal{H}_{m_n}$.  This
  implies that~$Y_0$ defined above is closed under positively
  homogeneous continuous function calculus. It follows in particular
  that~$Y_0$ is a sublattice of $\mathbb R^{B_{E^*}}$ because the
  functions $(s,t)\mapsto s\vee t$ and $(s,t)\mapsto s+\alpha t$ for
  $\alpha\in\mathbb R$ both belong to~$\mathcal{H}_2$.

  Let $\mathcal{N}_{\mathcal{D}}$ denote the collection of all
  $\mathcal D$-convex lattice seminorms $\nu\colon Y_0\to[0,\infty)$
  which satisfy
  \begin{equation}\label{Eq:nuofxhat}
    \nu(\delta_{x})\le\norm{x}
  \end{equation}
  for every $x\in E$, and define
  \begin{equation}\label{Eq:DnormY0}
    \norm{f}_{\mathcal{D}}
    =\sup\bigl\{\nu(f)\mid \nu\in\mathcal{N}_{\mathcal{D}}\bigr\}
  \end{equation}
  for $f\in Y_0$. We claim that this quantity is finite. Indeed,
  \eqref{defn:Y0} implies that $f=h(\delta_{x_1},\dots,\delta_{x_m})$
  for some $m\in\mathbb N$, $h\in\mathcal{H}_m$, and
  $x_1,\dots,x_m\in E$.  Then, applying
  \Cref{NJLlemma2Aug}\eqref{NJLlemma2AugIII}, we obtain
  $\norm{f}_e\le\norm{h}_{\mathcal H_m}$, where
  $e=\bigvee_{i=1}^m\abs{\delta_{x_i}}$. By the
  definition~\eqref{eq:IeNorm} of the norm~$\norm{\,\cdot\,}_e$,
  this means that
  \begin{displaymath}
    \abs{f}\le\norm{h}_{\mathcal{H}_m}e
    \le\norm{h}_{\mathcal{H}_m}\sum_{i=1}^m\abs{\delta_{x_i}}, 
  \end{displaymath}
  and therefore we have
  \begin{displaymath}
    \nu(f)\le\norm{h}_{\mathcal{H}_m}\sum_{i=1}^m
    \nu\bigl(\abs{\delta_{x_i}}\bigr)
    \le\norm{h}_{\mathcal{H}_m}\sum_{i=1}^m\norm{x_i}
  \end{displaymath}
  for every $\nu\in\mathcal{N}_{\mathcal{D}}$.  Since the right-hand
  side of this estimate is independent of~$\nu$, we conclude that
  $\norm{f}_{\mathcal D}$ is finite, as claimed.
      
  Consequently, being the supremum over a family of
  $\mathcal{D}$-convex lattice seminorms,
  $\norm{\,\cdot\,}_{\mathcal{D}}$ is itself a $\mathcal{D}$-convex
  lattice seminorm. It is in fact a norm, as we shall show next.  For
  every $x^*\in B_{E^*}$, we can define a lattice seminorm~$\nu_{x^*}$
  on~$Y_0$ by
  \begin{equation}\label{Defn:nux*}
    \nu_{x^*}(f) =\bigabs{f(x^*)}.
  \end{equation}
  Clearly~$\nu_{x^*}$ satisfies~\eqref{Eq:nuofxhat}.  Moreover,
  using~\eqref{DconvexExistsEq1}, we find
  \begin{displaymath}
    \nu_{x^*}\bigl(g(f_1,\ldots,f_m)\bigr)
    =\Bigabs{g\bigl(f_1(x^*),\ldots,f_m(x^*)\bigr)}
    =g\bigl(\nu_{x^*}(f_1),\ldots,\nu_{x^*}(f_m)\bigr)  
  \end{displaymath}
  for every $m\in\mathbb N$, $g\in\mathcal{H}_m^{>0}$ and
  $f_1,\ldots,f_m\in (Y_0)_+$. This implies that~$\nu_{x^*}$ is
  $\mathcal{D}$\nobreakdash-con\-vex because the constant $M(g)$
  in~\eqref{Defn:Dconvex} is at least~$1$, and therefore
  $\nu_{x^*}\in\mathcal{N}_\mathcal{D}$.
  
  Suppose that $f\in Y_0$ with $\norm{f}_{\mathcal{D}}=0$. Then
  $0=\nu_{x^*}(f) = \bigabs{f(x^*)}$ for every $x^*\in B_{E^*}$, so
  $f=0$.  Hence~$\norm{\,\cdot\,}_{\mathcal{D}}$ is a
  $\mathcal{D}$-convex lattice norm on~$Y_0$.  We can now define
  $\fbl^{\mathcal{D}}[E]$ as the com\-ple\-tion of~$Y_0$
  with respect to this norm.

  Lemma~\ref{lemma:DconvexCompletion} shows that
  $\fbl^{\mathcal{D}}[E]$ is a $\mathcal{D}$-convex
  Banach lattice. Moreover, the map
  $\phi_E^{\mathcal{D}}\colon\ E\to
  \fbl^{\mathcal{D}}[E]$ defined by
  $\phi_E^{\mathcal{D}}(x) =\delta_{x}$ is clearly linear. To see that
  it is iso\-metric, for $x\in E$, choose $x^*\in B_{E^*}$ such that
  $x^*(x) = \norm{x}$. Then, using~\eqref{Eq:nuofxhat},
  \eqref{Eq:DnormY0}, and~\eqref{Defn:nux*}, we obtain
  \begin{displaymath}
     \norm{\delta_{x}}_{\mathcal{D}}\le\norm{x}
     =\nu_{x^*}(\delta_{x})\le\norm{\delta_{x}}_{\mathcal{D}},
  \end{displaymath}
  so that $\bignorm{\phi_E^{\mathcal{D}}(x)}_{\mathcal{D}}=\norm{x}$. 

  It remains to verify the universal property. Let $T\colon E\to X$ be
  an operator into a $\mathcal{D}$-convex Banach lattice~$X$.  We may
  suppose that $\norm{T}=1$. Recall from the Introduction that
  the original construction of~$\fbl[E]$
  in~\cite[Section~2]{ART} defines it as a certain sublattice of
  positively homogeneous, real-valued functions defined on~$E^*$.
  Since such a function is uniquely determined by its action
  on~$B_{E^*}$, we may regard~$\fbl[E]$ as a sublattice
  of~$\mathbb R^{B_{E^*}}$ simply by restricting its elements
  to~$B_{E^*}$.  The universal property of~$\fbl[E]$
  means that there is a unique lattice homomorphism
  $S\colon\fbl[E]\to X$ such that $S(\delta_{x}) = Tx$
  for every $x\in E$, and $\norm{S}=1$. (This lattice
  homomorphism is usually denoted ~$\widehat{T}$; we use~$S$ here to
  avoid confusion with the lattice homomorphism that we seek to
  construct.)

  We observe that $\fbl[E]$ contains~$Y_0$ because
  $\delta_x\in\fbl[E]$ for every $x\in E$ and $\fbl[E]$
  is closed under positively homogeneous continuous function calculus
  because it is a Banach lattice. Hence we may consider the
  restriction $S_0\colon Y_0\to X$ of the lattice homomorphism~$S$
  to~$Y_0$. We claim that~$S_0$ is bounded with operator norm at most
  one. To verify this, we observe that $\nu(f) = \norm{S_0f}$ defines
  a lattice seminorm~$\nu$ on~$Y_0$ which
  satisfies~\eqref{Eq:nuofxhat} because
  \begin{displaymath}
    \nu(\delta_x) = \bignorm{S_0(\delta_x)}=\norm{Tx}\le\norm{x}
  \end{displaymath}
  for every $x\in E$. To show that~$\nu$ is $\mathcal{D}$-convex,
  write $\mathcal{D} = (\mathcal{G},M,\vartheta)$, and let
  $m\in\mathbb N$, $g\in\mathcal{G}_m$, and
  $f_1,\ldots,f_m\in (Y_0)_+$, where we assume that $f_1,\ldots,f_m$
  are pairwise disjoint if $\vartheta(g)=0$; note that in that
  case $S_0f_1,\ldots,S_0f_m$ are also pairwise disjoint. Therefore,
  using \Cref{CorFunctionCalcCommutesWithLatHom} and the
  $\mathcal{D}$-convexity of~$X$, we obtain
  \begin{align*}
    \nu\bigl(g(f_1,\ldots,f_m)\bigr) &=
        \Bignorm{S_0\bigl(g(f_1,\ldots,f_m)\bigr)}
        =\bignorm{g(S_0f_1,\ldots,S_0f_m)}\\
     &\le M(g)\cdot g\bigl(\norm{S_0f_1},\ldots,\norm{S_0f_m}\bigr)
       = M(g)\cdot g\bigl(\nu(f_1),\ldots,\nu(f_m)\bigr).
  \end{align*}
  It follows that $\nu\in\mathcal{N}_{\mathcal{D}}$, and therefore
  $\norm{f}_{\mathcal{D}}\ge\nu(f)=\norm{S_0f}$ for every
  $f\in Y_0$, which proves the claim.

  Hence~$S_0$ extends uniquely to a lattice homomorphism
  $\widehat{T}\colon\fbl^{\mathcal{D}}[E]\to X$, and
  $\norm{\widehat{T}}=\norm{S_0}\le 1$. We have
  $\widehat{T}\circ\phi_E^{\mathcal{D}}=T$ because
  $\widehat{T}(\delta_x)=S_0(\delta_x)=Tx$ for every $x\in
  E$. This implies in particular that
  $\norm{\widehat{T}}\ge\norm{T}=1$, so
  $\norm{\widehat{T}}= 1$.

  Finally, to prove the uniqueness of~$\widehat{T}$, suppose that
  $U\colon\fbl^{\mathcal{D}}[E]\to X$ is any lattice
  homomorphism satisfying $U(\delta_x) = Tx$ for every $x\in E$. Then
  \Cref{CorFunctionCalcCommutesWithLatHom} implies that
  \begin{align*}
    U\bigl(h(\delta_{x_1},\ldots,\delta_{x_m})\bigr) &=
      h\bigl(U(\delta_{x_1}),\ldots,U(\delta_{x_m})\bigr)\\
      &= h(Tx_1,\ldots,Tx_m)
      =\widehat{T}\bigl(h(\delta_{x_1},\ldots,\delta_{x_m})\bigr)
  \end{align*}
  for every $m\in\mathbb{N}$, $h\in\mathcal{H}_m$, and
  $x_1,\ldots,x_m\in E$, so~$U$ and~$\widehat{T}$ agree
  on~$Y_0$. Since~$Y_0$ is dense in~$\fbl^{\mathcal{D}}[E]$ and~$U$
  and~$\widehat{T}$ are bounded, we conclude that $U=\widehat{T}$.
\end{proof}

\begin{rem}\label{completions}
  It follows from general principles that the pair
  $\bigl(\fbl^\mathcal{D}[E], \phi_E^\mathcal{D}\bigr)$ constructed in
  \Cref{D-convex exists} is essentially unique.
\end{rem}

\begin{cor}\label{newCor21May} 
  Let $E$ be a Banach space and~$\mathcal{D}$  a convexity condition.
  \begin{enumerate}
  \item\label{newCor21May0} The sublattice generated by the set
    $\{\delta_{x}\mid x\in E\}$ is dense in~$\fbl^{\mathcal{D}}[E]$.
  \item\label{newCor21May1} Suppose that
    $S,T\colon \fbl^{\mathcal{D}}[E]\to X$ are lattice homomorphisms
    into a Banach lattice~$X$ satisfying
    $S\circ\phi_E^\mathcal{D} = T\circ\phi_E^\mathcal{D}$. Then $S=T$.
  \item\label{newCor21May2} $\fbl^\mathcal{D}[E]$ is separable if and
    only if~$E$ is separable.
  \end{enumerate}  
\end{cor}  

\begin{proof}    
  \eqref{newCor21May0}.  The closure of the sublattice
  of~$\fbl^{\mathcal{D}}[E]$ generated by $\{\delta_{x}\mid x\in E\}$
  is a Banach lattice and thus closed under positively homogeneous
  continuous function calculus. Hence it contains the sublattice~$Y_0$
  defined by~\eqref{defn:Y0}. This proves the claim because~$Y_0$ is
  dense in~$\fbl^{\mathcal{D}}[E]$ by definition.
    
  \eqref{newCor21May1}. This is immediate from~\eqref{newCor21May0}
  because lattice homomorphisms are automatically bounded.

  \eqref{newCor21May2}. This follows by combining~\eqref{newCor21May0}
  with \cite[p.~204, Exercise~9]{AB}.
\end{proof}  

\begin{example}
  Let $p\in[1,\infty]$. Taking
  $\mathcal{D}=(\mathcal{G}^p,M,\vartheta)$ for
  $\mathcal{G}^p=\bigl\{\norm{\,\cdot\,}_{\ell_p^m} \mid m\in
  \mathbb{N}\bigr\}$, $M\equiv C\in [1,\infty)$ and
  $\vartheta \equiv 1$ gives the \term{free $p$-convex Banach lattice
    with $p$\nobreakdash-con\-vex\-ity con\-stant~$C$} (cf.\
  \Cref{RemarkMainCasesofDconvex}). For $C=1$, we denote this space
  by $\fbl^{(p)}[E]$ and observe that, together with the map
  $\phi_E = \phi_E^{\mathcal{D}}$, it has the properties stated in
  \Cref{p-convex exists} because every $p$-convex Banach lattice~$X$
  can be renormed to have $p$\nobreakdash-con\-vex\-ity constant~$1$,
  with the new norm being $M^{(p)}(X)$-equivalent to the original
  norm. An explicit description of $\fbl^{(p)}[E]$ and its norm will
  be given in Sec\-tion~\ref{section:description}.

  Taking instead $\vartheta\equiv 0$ (and~$\mathcal{G}^p$ and
  $M\equiv C\in[1,\infty)$ as above), we obtain the \term{free Banach
    lattice satisfying upper $p$-estimates with constant~$C$}.
\end{example}

\section{Basic properties of $\fbl^\mathcal{D}[E]$}\label{Basic}

The aim of this section is to establish some basic properties of $\fbl^\mathcal{D}[E]$. 
Throughout, $\mathcal{D} = (\mathcal{G},M,\vartheta)$ denotes a convexity condition, and in si\-tua\-tions where no confusion can arise, we will write $\phi_E$ for the canonical map $\phi_E^\mathcal{D}\colon E\to \fbl^\mathcal{D}[E]$. 

\subsection*{Complementation} 
We begin with a generalization of \cite[Corollary 2.7]{ART}. Recall
that a Banach space~$F$ is $C$-isomorphic to a complemented subspace
of a Banach space~$E$ for some constant $C\ge 1$ if there are
operators $U\colon F\to E$ and $V\colon E\to F$ such that
$I_F = V\circ U$ and $\norm{U}\,\norm{V}\le C$, where $I_F$
denotes the identity operator on~$F$. In the case where $E$ and $F$
are Banach lattices, we say that $F$ is \term{$C$-lattice complemented
  in~$E$} if the operators $U$ and $V$ can be chosen to be lattice
homo\-mor\-phisms. Note that the condition $I_F = V\circ U$ implies
that $P:=U\circ V$ is idem\-potent, and~$U$ is an isomorphism of~$F$
onto the range of~$P$.

\begin{prop}
  Let $E$ and $F$ be Banach spaces, where $F$ is $C$-isomorphic
  to a complemented subspace of~$E$ for some constant $C\ge 1$. Then
  $\fbl^{\mathcal{D}}[F]$ is $C$-lattice  complemented in~$\fbl^{\mathcal{D}}[E]$.
\end{prop}

\begin{proof}
  Let $U\colon F\to E$ and $V\colon E\to F$ be operators such that
  $I_F = V\circ U$ and $\norm{U}\,\norm{V}\le C$. Since
  $\phi_E\circ U\colon F\to \fbl^{\mathcal{D}}[E]$ is an operator into
  a $\mathcal{D}$-convex Banach lattice, there is a unique lattice
  homomorphism
  $\widetilde{U}:=\widehat{\phi_E\circ U}\colon
  \fbl^{\mathcal{D}}[F]\to \fbl^{\mathcal{D}}[E]$ such that
  $\widetilde{U}\circ\phi_F = \phi_E\circ U$, and
  $\norm{\widetilde{U}}=\norm{\phi_E\circ U}=\norm{U}$ (the last
  equality follows because $\phi_E$ is an isometry). Similarly we
  obtain a unique lattice homomorphism
  $\widetilde{V}:=\widehat{\phi_F\circ V}\colon
  \fbl^{\mathcal{D}}[E]\to \fbl^{\mathcal{D}}[F]$ such that
  $\widetilde{V}\circ\phi_E = \phi_F\circ V$, and
  $\norm{\widetilde{V}}=\norm{V}$. Now we check that
  \begin{displaymath}
    \widetilde{V}\circ\widetilde{U}\circ\phi_F
    = \widetilde{V}\circ\phi_E\circ U
    =\phi_F\circ V\circ U = \phi_F,
  \end{displaymath}
  so $\widetilde{V}\circ\widetilde{U}=I_{\fbl^{\mathcal{D}}[F]}$ by
  \Cref{newCor21May}\eqref{newCor21May1}.
\end{proof}  

We next characterize when $\phi_E(E)$ is complemented in its free space, cf.\ \cite[Corollary~2.5]{ART}.

\begin{prop}\label{phi complemented}
  Let $E$ be a Banach space and $C\ge 1$. Then~$E$ is $C$-isomorphic
  to a complemented subspace of a $\mathcal{D}$-convex Banach lattice
  if and only if~$\phi_E (E)$ is $C$\nobreakdash-complemented in
  $\fbl^\mathcal{D}[E]$.
\end{prop}

\begin{proof}
  Suppose that $E$ is $C$-isomorphic to a complemented subspace of a
  $\mathcal{D}$-convex Banach lattice~$X$, so that $I_E=V\circ U$ for
  some operators $U\colon E\to X$ and $V\colon X\to E$ with
  $\norm{U}\,\norm{V}\le C$. Then the inclusion map
  $J\colon\phi_E(E)\to \fbl^\mathcal{D}[E]$ and the composite operator
  $W:=\phi_E\circ V\circ \widehat{U}\colon\fbl^\mathcal{D}[E]\to
  \phi_E(E)$ satisfy $W\circ J=I_{\phi_E (E)}$ and
  $\norm{W}\,\norm{J}\le C$, so $\phi_E(E)$ is
  $C$\nobreakdash-complemented in $\fbl^\mathcal{D}[E]$.

  The converse is immediate because~$E$ is isometric to~$\phi_E(E)$
  and~$\fbl^\mathcal{D}[E]$ is a $\mathcal{D}$-convex Banach lattice.
\end{proof}

\begin{rem}
  It is a famous open question whether every complemented subspace of
  a Banach lattice is isomorphic to a Banach lattice. \Cref{phi
    complemented} reduces this to a question about free Banach
  lattices, and extends the question to the $\mathcal{D}$-convex case.
\end{rem}

\subsection*{Projectivity}
We shall next study the projective objects in the category of $\mathcal{D}$\nobreakdash-con\-vex Banach lattices, beginning with a 
result which recovers and extends one of the main results of~\cite{LLOT}.

In line with general conventions, we say that a  Banach lattice~$Z$ is \term{projectively universal for the class of separable, $\mathcal{D}$-convex Banach lattices} if~$Z$ is separable and $\mathcal{D}$-convex, and every separable, $\mathcal{D}$-convex Banach lattice~$X$ is lattice isometric to a quotient of~$Z$. Note that this  is  equivalent to the existence of a lattice homomorphism from~$Z$ onto~$X$ which maps the open unit ball of~$Z$ onto the open unit ball of~$X$. 

\begin{thm}\label{Recover Denny}
  The Banach lattice $\fbl^{\mathcal{D}}[\ell_1]$ is projectively
  universal for the class of separable, $\mathcal{D}$-convex Banach
  lattices.
\end{thm}

\begin{proof}
  Let $X$ be a separable, $\mathcal{D}$-convex Banach lattice. Using
  the separable projective universality of~$\ell_1$, we can find a
  linear surjection $T\colon\ell_1\to X$ which maps the open unit ball
  of~$\ell_1$ onto the open unit ball of~$X$, and hence the induced
  lattice homomorphism
  $\widehat{T}\colon \fbl^{\mathcal{D}}[\ell_1]\to X$ maps the open
  unit ball of~$\fbl^{\mathcal{D}}[\ell_1]$ onto the open unit ball
  of~$X$. That $\fbl^{\mathcal{D}}[\ell_1]$ is separable follows from
  \Cref{newCor21May}\eqref{newCor21May2}.
\end{proof}

\begin{rem}
  In \Cref{Recover Denny}, one can of course replace~$\ell_1$ with any
  separable Banach space which has every separable,
  $\mathcal{D}$-convex Banach lattice as a quotient. For example, one
  can iterate the process to get that $\fbl\bigl[\fbl[\ell_1]\bigr]$ is
  projectively universal for the class of separable Banach lattices.
\end{rem}

\begin{rem}
  Similar arguments establish an analogous result for arbitrary
  density character~$\kappa$, replacing $\ell_1$ with
  $\ell_1(\kappa)$.
\end{rem}

We present the next simple lemma due to its relevance to \Cref{Recover
  Denny}, and because it will be needed for subsequent results. It
shows that the separable, $\mathcal{D}$\nobreakdash-con\-vex Banach
lattices are exactly the lattice quotients of
$\fbl^{\mathcal{D}}[\ell_1]$.

\begin{lem}\label{quotients}
  A quotient of a $\mathcal{D}$-convex Banach lattice by a closed
  ideal is $\mathcal{D}$-convex.
\end{lem}

\begin{proof}
  Let $J$ be a closed ideal of a $\mathcal{D}$-convex Banach
  lattice~$X$, and denote the quo\-tient homomorphism by
  $Q\colon X\to X/J$. Further, let $g\in \mathcal{G}_m$ for
  some $m\in\mathbb N$, and take
  $\varphi_1,\dots,\varphi_m\in (X/J)_+$, where we suppose that
  $\varphi_1,\dots,\varphi_m$ are pairwise disjoint if
  $\vartheta (g)=0$.  For $\varepsilon>0$, we can choose
  $x_1,\ldots,x_m\in X_+$ such that
  $\norm{x_k}\le (1+\varepsilon)\norm{\varphi_k}$ and $Qx_k=\varphi_k$ for
  each~$k$, and we can also arrange that $x_1,\ldots,x_m$ are pairwise
  disjoint if $\vartheta (g)=0$ (cf.\ \cite[Section 9]{dePW}).  Then,
  using \Cref{CorFunctionCalcCommutesWithLatHom}, we have
  \begin{multline*}
    \bignorm{g(\varphi_1,\dots,\varphi_m)}
    =\Bignorm{Q\bigl(g(x_1,\dots,x_m)\bigr)}
    \le M(g)\cdot g\bigl(\norm{x_1},\dots,\norm{x_m}\bigr)\\ 
    \le M(g)\cdot g\Bigl((1+\varepsilon)\norm{\varphi_1},
        \dots,(1+\varepsilon)\norm{\varphi_m}\Bigr)
        \xrightarrow{\varepsilon\to 0}
   M(g)\cdot g\bigl(\norm{\varphi_1},\dots,\norm{\varphi_m}\bigr),
  \end{multline*}
  and the conclusion follows.
\end{proof}

The notion of a \term{projective} Banach lattice was introduced in
\cite{dePW}. Informally, a Banach lattice $P$ is projective if every
lattice homomorphism from $P$ to a quotient of a Banach lattice $X$
can be lifted to a lattice homomorphism into $X$, with control of the
norm. As a consequence of the fact that $\ell_1(A)$ is a projective
Banach space for any non\-empty set~$A$; it was shown in \cite{dePW}
that $\fbl\bigl[\ell_1(A)\bigr]$ is a projective Banach lattice. Other
examples of projective Banach lattices include all finite dimensional
Banach lattices \cite{dePW} and $C(K)$ spaces for every compact
Hausdorff space $K$ which is an absolute neighbourhood retract; see
\cite{dePW,AMA}, as well as \cite{AMA2} for more recent related
results.

To conclude this section, we find a Banach space property of $E$ which
characterizes when $\fbl[E]$ is projective. Note in this connection
that it was shown in \cite{AMA} that if $\fbl[E]$ is projective, then
necessarily $E$ has the Schur property. Since we now also have the
spaces $\fbl^\mathcal{D}[E]$ ---~and our characterizations extend to
these spaces in the appropriate way~--- we introduce two new
definitions. Part~\eqref{Pdef1i} extends the definition of projectivity
in \cite{dePW}, while~\eqref{Pdef1ii} is a variant that makes sense for
arbitrary Banach spaces.

\begin{defn}\label{Pdef1}
  \begin{enumerate}
  \item\label{Pdef1i} A Banach lattice $P$ is \term{projective for
      $\mathcal{D}$-convex Banach lattices} if, for every closed
    ideal~$J$ of a $\mathcal{D}$-convex Banach lattice~$X$, every
    lattice homomorphism $T\colon P\to X/J$, and every
    $\varepsilon>0$, there is a lattice homomorphism
    $\widehat{T}\colon P\to X$ such that $T=Q\circ \widehat{T}$ and
    $\norm{\widehat{T}}\le (1+\varepsilon)\norm{T}$, where
    $Q\colon X\to X/J$ denotes the quotient homomorphism.
  \item\label{Pdef1ii} A Banach space $E$ is \term{linearly projective
      for $\mathcal{D}$-convex Banach lattices} if, for every closed
    ideal~$J$ of a $\mathcal{D}$-convex Banach lattice~$X$, every
    operator $T\colon E\to X/J$, and every $\varepsilon>0$, there is
    an operator $\widehat{T}\colon E\to X$ such that
    $T=Q\circ \widehat{T}$ and
    $\norm{\widehat{T}}\le (1+\varepsilon)\norm{T}$, again with
    $Q\colon X\to X/J$ de\-noting the quotient homomorphism.
  \end{enumerate}
\end{defn}

Note that in the above definitions we require that $X$ is
$\mathcal{D}$-convex and this implies that $X/J$ is as well, by
\Cref{quotients}. For convenience, given convexity conditions
$\mathcal{D}$ and~$\mathcal{D}'$, let us write
$\mathcal{D}'\le\mathcal{D}$ whenever $\mathcal{D}$-convexity implies
$\mathcal{D'}$-convexity.  The following result (applied with
$\mathcal{D}=\mathcal{D}'$) clarifies the relationship between these
two notions, as well as a third ``hybrid'' notion.
 
\begin{prop}\label{Projliftstofbl}
  Suppose that~$\mathcal{D}$ and~$\mathcal{D}'$ are convexity
  conditions with $\mathcal{D}'\leq\mathcal{D}$. Then the following
  three conditions are equivalent for a Banach space~$E\colon$
  \begin{enumerate}
  \item\label{Projliftstofbl1} $E$ is linearly projective for
    $\mathcal{D}$-convex Banach lattices.
  \item\label{Projliftstofbl2} $\fbl^\mathcal{D'}[E]$ is projective
    for $\mathcal{D}$-convex Banach lattices.
  \item\label{Projliftstofbl3} For every closed ideal~$J$ of a
    $\mathcal{D}$-convex Banach lattice~$X,$ every lattice
    homomorphism $T\colon \fbl^\mathcal{D'}[E]\to X/J,$ and every
    $\varepsilon>0,$ there is an operator
    $\widehat{T}\colon \fbl^\mathcal{D'}[E]\to X$ such that
    $T=Q\circ \widehat{T}$ and
    $\norm{\widehat{T}}\le(1+\varepsilon)\norm{T},$ where
    $Q\colon X\to X/J$ denotes the quotient homomorphism.
  \end{enumerate}
\end{prop}

\begin{proof}
  \eqref{Projliftstofbl1}$\Rightarrow$\eqref{Projliftstofbl2}: Suppose
  that $E$ is linearly projective for $\mathcal{D}$-convex Banach
  lattices, and let $T\colon \fbl^\mathcal{D'}[E]\to X/J$ be a lattice
  homomorphism, where~$J$ is a closed ideal of a $\mathcal{D}$-convex
  Banach lattice~$X$.  By the hypothesis, for every $\varepsilon>0$,
  we can lift the operator
  $S:=T\circ\phi_E^{\mathcal{D}'}\colon E \to X/J$ to an operator
  $\widehat{S}\colon E\to X$ with $Q\circ \widehat{S}=S$ and
  $\norm{\widehat{S}}\le(1+\varepsilon)\norm{S}$, where
  $Q \colon X\to X/J$ is the quotient homomorphism. Since~$X$ is
  $\mathcal{D'}$\nobreakdash-con\-vex, \Cref{D-convex exists} implies
  that~$\widehat{S}$ lifts to a lattice homomorphism
  $\widehat{T}\colon \fbl^\mathcal{D'}[E]\to X$ with
  $\widehat{T}\circ\phi^{\mathcal{D}'}_E = \widehat{S}$ and
  $\norm{\widehat{T}}=\norm{\widehat{S}}$. We check
  that~$\widehat{T}$ has the required properties:
  $Q\circ\widehat{T} = T$ by \Cref{newCor21May}\eqref{newCor21May1}
  because
  \begin{displaymath}
    Q\circ\widehat{T}\circ\phi_E^{\mathcal{D}'}
    =Q\circ\widehat{S} = S =T\circ\phi_E^{\mathcal{D}'},   
  \end{displaymath}
  and
  \begin{math}
    \norm{\widehat{T}}\le(1+\varepsilon)\norm{S}
    =(1+\varepsilon)\bignorm{T\circ\phi_E^{\mathcal{D}'}}
    \le(1+\varepsilon)\norm{T}.
  \end{math}
    
  \eqref{Projliftstofbl2}$\Rightarrow$\eqref{Projliftstofbl3} is obvious.
    
  \eqref{Projliftstofbl3}$\Rightarrow$\eqref{Projliftstofbl1}: Suppose
  that~\eqref{Projliftstofbl3} is satisfied, and let
  $T\colon E\to X/J$ be an operator, where~$J$ is a closed ideal of a
  $\mathcal{D}$\nobreakdash-con\-vex Banach lattice~$X$. Using
  \Cref{quotients} and \Cref{D-convex exists}, we can find a lattice
  homomorphism $S\colon\fbl^\mathcal{D'}[E]\to X/J$ with
  $S\circ\phi_E^{\mathcal{D}'} = T$ and
  $\norm{S}=\norm{T}$. By~\eqref{Projliftstofbl3}, for every
  $\varepsilon>0$, there is an operator
  $\widehat{S}\colon\fbl^\mathcal{D'}[E]\to X$ such that
  $S=Q\circ\widehat{S}$ and
  $\norm{\widehat{S}}\le(1+\varepsilon)\norm{S}$. Then the operator
  $\widehat{T}\colon=\widehat{S}\circ \phi_E^{\mathcal{D}'}\colon
  E\to X$ has the desired properties: Its norm is at most
  $(1+\varepsilon)\norm{T}$ and
  \begin{displaymath}
    Q\circ\widehat{T} =  Q\circ \widehat{S}\circ \phi_E^{\mathcal{D}'}
    = S\circ \phi_E^{\mathcal{D}'} = T. \qedhere
  \end{displaymath}
\end{proof}

\begin{rem}
  One may now wonder when $\fbl[E]$ is linearly projective for Banach
  lattices. This will essentially never happen. Indeed, if it were
  then $\fbl[E]$ would have the Schur property, so in particular would
  be order continuous. However, $\fbl[E]$ will not be order continuous
  as long as $\dim E>1$.  It is also not true that linear projectivity
  implies lattice projectivity: $\ell_1(A)$ is a linearly projective
  Banach space, but it follows from \cite[Corollary 10.5]{dePW} that
  it is not a projective Banach lattice when $A$ is uncountable.
\end{rem}

\section{Free AM-spaces}\label{AMstuff}

An \term{AM-space} is a Banach lattice~$X$ for which $\norm{x\vee
y}=\norm{x}\vee\norm{y}$ whenever $x,y\in
X_+$ are disjoint.  A \term{unital AM-space} is a
non\-zero AM-space~$X$ which contains a positive element~$e$ such that
$I_e=X$ and the norm~$\norm{\,\cdot\,}_e$ defined
by~\eqref{eq:IeNorm} is equal to the given norm on~$X$.

Kakutani's famous representation theorem for AM-spaces states that a
Banach lattice is an AM-space if and only if it admits an isometric
lattice homomorphism into~$C(K)$ for some compact Hausdorff space~$K$,
and it is a unital AM-space if and only if it is isometrically lattice
isomorphic to~$C(K)$ for some~$K$ (see, e.g.,
\cite[Theorem~1.b.6]{LT2}).

We begin this section by identifying the convexity
conditions~$\mathcal{D}$ which correspond to AM-spaces, and we then
show that, for a given Banach space~$E$, they all give rise to the
same $\mathcal{D}$-convex free Banach lattice.  For the avoidance of
any doubt in the following definition, recall that
$\bignorm{(s,t)}_{\ell_\infty^2} = \abs{s}\vee\abs{t}$ for
$s,t\in\mathbb R$.

\begin{defn} 
  An \term{AM-convexity condition} is a convexity condition
  $\mathcal{D} = (\mathcal{G},M,\vartheta)$ for which
  $\norm{\,\cdot\,}_{\ell_\infty^2}\in\mathcal{G}$ with
  $M(\norm{\,\cdot\,}_{\ell_\infty^2})=1$.
\end{defn}
  
\begin{lem}\label{LemmaAMchar}
  The following conditions are equivalent for a Banach lattice~$X\colon$
  \begin{enumerate}
  \item\label{LemmaAMchar1} $X$ is an AM-space.
  \item\label{LemmaAMchar3} $X$ is $\mathcal{D}$-convex for every
    convexity condition $\mathcal{D}$.
  \item\label{LemmaAMchar2} $X$ is $\mathcal{D}$-convex for some
    AM-convexity condition $\mathcal{D}$.
  \end{enumerate}  
\end{lem}

\begin{proof}
  \eqref{LemmaAMchar1}$\Rightarrow$\eqref{LemmaAMchar3}: Suppose
  that~$X$ is an AM-space, and take an isometric lattice homomorphism
  $T\colon X\to C(K)$ for some compact Hausdorff space~$K$. Let
  $m\in\mathbb N$, $g\in\mathcal{H}_m^{>0}$, and
  $x_1,\ldots,x_m\in X_+$.  Since
  $0\le (Tx_j)(t)\le\norm{Tx_j}_\infty =\norm{x_j}$ for
  every $t\in K$ and $j=1,\ldots,m$, we have
  \begin{displaymath}
    g\bigl(\norm{x_1},\ldots,\norm{x_m}\bigr)
    \ge g\bigl((Tx_1)(t),\ldots,(Tx_m)(t)\bigr)
    = T\bigl(g(x_1,\ldots,x_m)\bigr)(t)\ge 0,
  \end{displaymath}
  where the equality follows from
  \Cref{CorFunctionCalcCommutesWithLatHom} and the fact that the
  positively homogeneous continuous function calculus is defined
  pointwise in~$C(K)$. Taking the supremum over all $t\in K$, we
  obtain
  \begin{displaymath}
    g\bigl(\norm{x_1},\ldots,\norm{x_m})
    \ge\Bignorm{T\bigl(g(x_1,\ldots,x_m)\bigr)}_\infty
    =\bignorm{g(x_1,\ldots,x_m)},
  \end{displaymath}
  which shows that~\eqref{Defn:Dconvex} is satisfied because
  $M(g)\ge 1$. Hence $X$ is $\mathcal{D}$-convex, no matter which
  convexity condition~$\mathcal{D}$ we consider.

  \eqref{LemmaAMchar3}$\Rightarrow$\eqref{LemmaAMchar2} is trivial. 
  
  \eqref{LemmaAMchar2}$\Rightarrow$\eqref{LemmaAMchar1}: Suppose that
  $X$ is $\mathcal{D}$-convex for some AM-convexity
  condition~$\mathcal{D}$. Then we have
  $\norm{x\vee y}\le\norm{x}\vee\norm{y}$ whenever
  $x,y\in X_+$ are disjoint. The opposite inequality is always true by
  monotonicity of the norm, so~$X$ is an AM-space.
\end{proof}

Let $E$ be a Banach space and $\mathcal{D}$ a convexity condition.
\Cref{newCor21May}\eqref{newCor21May0} shows that we may view
$\fbl^\mathcal{D}[E]$ as the completion of the sublattice~$L$
of~$\mathbb R^{B_{E^*}}$ generated by the set
$\{\delta_{x}\mid x\in E\}$ with respect to the
norm~$\norm{\,\cdot\,}_{\mathcal{D}}$ given
by~\eqref{Eq:DnormY0}. In fact $L\subseteq C(B_{E^*})$,
where~$B_{E^*}$ is equipped with the relative weak$^*$ topology,
because~$\delta_x$ is weak$^*$ continuous for every $x\in E$, and
using the seminorms~$\nu_{x^*}$ defined by~\eqref{Defn:nux*} for
$x^*\in B_{E^*}$, we see that
$\norm{f}_\infty\le\norm{f}_\mathcal{D}$ for every
$f\in L$.  Hence the inclusion map
\begin{math}
  \bigl(L,\norm{\,\cdot\,}_{\mathcal{D}}\bigr)\to
  \bigl(\overline{L}^{\norm{\,\cdot\,}_\infty},\norm{\,\cdot\,}_\infty\bigr)
\end{math}
extends to a lattice homomorphism of norm at most~$1$ defined
on~$\fbl^{\mathcal{D}}[E]$. Despite this, is not clear if this map is
injective --- we do not even know this for the free Banach lattice
satisfying an upper $p$-estimate with constant~$1$. It is, however, an
isometric isomorphism provided that~$\mathcal{D}$ is an AM-convexity
condition, as we shall prove next. For that reason, we
term~$\overline{L}^{\norm{\,\cdot\,}_\infty}$ the \term{free AM-space
  over}~$E$.

\begin{thm}\label{freeAM}
  Let $E$ be a Banach space and~$\mathcal{D}$ an AM-convexity
  condition. Then $\fbl^{\mathcal{D}}[E]$ is isometrically lattice
  isomorphic to the $\norm{\,\cdot\,}_\infty$\nobreakdash-closed
  sublattice of $C(B_{E^*})$ generated by $\{\delta_{x}\mid x\in E\}$.
\end{thm}

\begin{proof}
  By the above remarks, it suffices to show that
  $\norm{f}_\infty\ge\norm{f}_\mathcal{D}$ for every
  $f\in L$.  Write~$f$ as
  $f=\bigvee_{j=1}^n\delta_{x_j}-\bigvee_{j=1}^n\delta_{y_j}$,
  where $n\in\mathbb N$ and $x_1,\ldots,x_n,y_1,\ldots,y_n\in E$,
  using~\eqref{eq:ABsublattice}. Since $\fbl^{\mathcal{D}}[E]$ is
  $\mathcal{D}$-convex, \Cref{LemmaAMchar} implies that it is an
  AM-space, so we can find an isometric lattice homomorphism
  $U\colon\fbl^{\mathcal{D}}[E]\to C(K)$ for some compact Hausdorff
  space~$K$. For $t\in K$, let $\eta_t\in B_{C(K)^*}$ be the
  evaluation functional at~$t$, and define
  $x^*=\bigl(U\circ\varphi_E^{\mathcal{D}}\bigr)^*\eta_t\in B_{E^*}$. Then we
  have
  \begin{displaymath}
    \delta_x(x^*)
    =\Bigl\langle x, \bigl(U\circ\varphi_E^{\mathcal{D}}\bigr)^*\eta_t\Bigr\rangle
    =\Bigl\langle\bigl(U\circ\varphi_E^{\mathcal{D}}\bigr)x,\eta_t\Bigr\rangle
    = (U\delta_x)(t) 
  \end{displaymath}
  for every $x\in E$, so that
  \begin{displaymath}
    f(x^*)
    =\bigvee_{j=1}^n\delta_{x_j}(x^*) - \bigvee_{j=1}^n\delta_{y_j}(x^*)
    =\bigvee_{j=1}^n(U\delta_{x_j})(t) - \bigvee_{j=1}^n(U\delta_{y_j})(t)
    = (Uf)(t).
  \end{displaymath}
  It follows that
  \begin{math}
    \norm{f}_\infty\ge\sup_{t\in K}\bignorm{(Uf)(t)}
    =\norm{Uf}_\infty =\norm{ f}_{\mathcal{D}},
  \end{math}
  as required.
\end{proof}

Our next result complements \Cref{freeAM} by identifying the free
unital AM-space over a Banach space~$E$. More precisely, in the light
of Kakutani's representation theorem for unital AM-spaces stated
above, it can be paraphrased as saying that the pair
$\bigl(C(B_{E^*}),\phi_E\bigr)$ is the free unital AM-space over~$E$, where
$B_{E^*}$ is equipped with the relative weak$^*$ topology and
$\phi_E\colon E\to C(B_{E^*})$ denotes the linear isometry given by
$\phi_E(x) = \delta_x$, as usual.

\begin{thm}\label{t:1AM}
  Let $E$ be a Banach space. For every compact Hausdorff space~$K$ and
  every non\-zero operator $T\colon E\to C(K)$, there exists a unique
  lattice homomorphism $\widehat{T}\colon C(B_{E^*})\to C(K)$ such
  that $\widehat{T}\circ\phi_E=T$ and
  $\widehat{T}\mathbb{1}=\mathbb{1}$, where~$\mathbb{1}$ denotes the
  constant function~$1$. Moreover, $\norm{\widehat{T}}=\norm{T}$ and
  $\widehat{T}$ is an algebra homomorphism.
\end{thm}

\begin{proof}
  We may suppose that $\norm{T}=1$. Then the map
  $t\mapsto T^*\eta_t$, where~$\eta_t$ is the evaluation functional
  at~$t$, maps~$K$ into~$B_{E^*}$, and it is continuous with respect
  to the relative weak$^*$ topology on~$B_{E^*}$, so we can define a
  map $\widehat{T}\colon C(B_{E^*})\to C(K)$ by
  $\widehat{T}(f)(t) = f(T^*\eta_t)$ for $f\in C(B_{E^*})$ and
  $t\in K$.  Since the algebraic and lattice operations in
  both~$C(B_{E^*})$ and~$C(K)$ are defined pointwise, it is easy to
  check that $\widehat{T}$ is a lattice and algebra homomorphism with
  $\widehat{T}\mathbb{1}=\mathbb{1}$ (see also \cite[Theorem
  3.2.12]{MN} for a more global picture of these maps). Moreover, we
  have
  \begin{displaymath}
    \bigl(\widehat{T}\circ\phi_E\bigr)(x)(t)
    =\delta_x(T^*\eta_t)=\langle x,T^*\eta_t\rangle = (Tx)(t)    
  \end{displaymath}
  for every $x\in E$ and $t\in K$, so that
  $\widehat{T}\circ\phi_E = T$. This implies in particular that
  $\norm{\widehat{T}}\ge\norm{T}=1$. On the other hand,
  $\bigabs{(\widehat{T}f)(t)}= \bigabs{f(T^*\eta_t)}\le\norm{f}_\infty$
  for every $t\in K$ and $f\in C(B_{E^*})$, so that
  $\norm{\widehat{T}f}_\infty\le\norm{f}_\infty$, and therefore
  $\norm{\widehat{T}}=1$.
  
  Finally, to prove uniqueness, suppose that
  $U\colon C(B_{E^*})\to C(K)$ is any lattice homo\-mor\-phism
  satisfying $U\circ\phi_E=T$ and $U\mathbb{1}=\mathbb{1}$. Then
  $\widehat{T}$ and~$U$ agree on the sub\-lattice of~$C(B_{E^*})$
  generated by $\{\delta_x \mid x\in E\}\cup \{\mathbb{1}\}$.  The
  Stone--Weier\-strass Theorem implies that this sublattice is dense
  in~$C(B_{E^*})$, and therefore, being bounded, $\widehat{T}$ and~$U$
  are equal.
 \end{proof}

\section{An explicit formula for the norm of the free p-convex Banach lattice}\label{section:description}

The aim of this section is to verify the explicit
formula~\eqref{Eq:FBLpNorm} for the norm of the free
$p$\nobreakdash-con\-vex Banach lattice $\fbl^{(p)}[E]$. Throughout,
$p\in (1,\infty)$, $E$~is a Banach space, $H[E]$~denotes the vector
lattice of all positively homogeneous functions
$E^*\to\mathbb R$, $\norm{f}_p$ is defined by~\eqref{Eq:FBLpNorm}
for every $f\in H[E]$, and~$L$ denotes the sublattice of~$H[E]$
generated by the evaluation maps~$\delta_x$ for $x\in E$. Note that
this definition of~$L$ differs slightly from the one we used in the
previous section, where the functions in~$L$ were defined
on~$B_{E^*}$, not~$E^*$. However, as already remarked in the proof of
\Cref{D-convex exists}, this difference is purely formal because a
positively homogeneous function $E^*\to\mathbb R$ is uniquely
determined by its action on~$B_{E^*}$.

To simplify notation, we write 
\begin{equation*}
  \weaksumnorm{p}{x_1^*,\dots,x_n^*}
  =\sup_{x\in B_E}\Bigl(\sum_{k=1}^n \abs{x_k^*(x)}^p\Bigr)^{\frac{1}{p}}
\end{equation*}
for  the weak $p$-summing norm of a finite sequence $(x_k^*)_{k=1}^n$ in~$E^*$. 

Let us begin by trying to motivate the expression~\eqref{Eq:FBLpNorm}
for the norm of $\fbl^{(p)}[E]$.  Consider an operator
$T\colon E\to\ell_p^n$ for some $n\in\mathbb N$. Writing
$(e_k)_{k=1}^n$ for the unit vector basis of $\ell_p^n$, we can
express~$T$ as
\begin{displaymath}
 T(x)=\sum_{k=1}^n x_k^*(x) e_k  
\end{displaymath}
for a certain finite sequence $(x_k^*)_{k=1}^n$ in $E^*$ and every
$x\in E$, and we have $\norm{T}= \weaksumnorm{p}{x_1^*,\ldots,x_n^*}$
in the notation introduced above.  (In fact $x_k^*=T^*e_k^*$, but
this formula will not be helpful for our purposes.) It is easy to
check that the only way to extend~$T$ to a lattice homomorphism
$\widehat{T}\colon L \to\ell_p^n$ is by defining
\begin{displaymath}
 \widehat{T}f=\sum_{k=1}^n f(x_k^*)e_k
\end{displaymath}
for every $f\in L$.  Thus, we must have
\begin{displaymath}
  \Bigl(\sum_{k=1}^n \bigabs{f(x_k^*)}^p\Bigr)^{\frac1p}
  =\norm{\widehat{T}f}_{\ell_p^n}
  \le\norm{\widehat{T}}\,\norm{f}_{\fbl^{(p)}[E]}
  =\weaksumnorm{p}{x_1^*,\ldots,x_n^*}\,\norm{f}_{\fbl^{(p)}[E]}.  
\end{displaymath}
Taking the supremum over all possible choices of the operator~$T$
subject to $\norm{T}\le 1$, we conclude that $\norm{f}_p$ defined
by~\eqref{Eq:FBLpNorm} satisfies the inequality
$\norm{f}_p\le\norm{f}_{\fbl^{(p)}[E]}$. Morally speaking,
establishing equality of these two norms means that extending
operators into arbitrary $p$-convex Banach lattices can in a certain
sense be reduced to the extension of operators into the spaces
$\ell_p^n$ for $n\in\mathbb N$.\smallskip

We now turn to the explicit description of $\fbl^{(p)}[E]$. It is easy  to see that 
\begin{displaymath}
  H_p[E]:=\bigl\{f\in H[E] \mid  \norm{f}_{p} <\infty\bigr\}
\end{displaymath}
is a sublattice of~$H[E]$ and that~$\norm{\,\cdot\,}_p$ defines a
complete $p$-convex lattice norm on~$H_p[E]$ with
$p$\nobreakdash-con\-vex\-ity constant one. Moreover,
$\norm{\delta_x}_p=\norm{x}$ for every $x\in E$,
so~$H_p[E]$ contains~$L$ as a sublattice. Hence we can define
$\fbl_p[E]$ as the closure of~$L$ in~$H_p[E]$, and the map
$\phi_E\colon E\to\fbl_p[E]$ given by $\phi_E(x)=\delta_x$ is a
linear isometry. Note the position of the index~$p$: We
write~$\fbl_p[E]$ for the Banach lattice that we have just defined to
distinguish it from the previously defined Banach
lattice~$\fbl^{(p)}[E]$. However, our next theorem will identify the
pair $(\fbl_p[E],\phi_E)$ as the free $p$-convex Banach lattice
generated by $E$, so once we have proved it, this distinction will no
longer be necessary.

\begin{thm}\label{t:fblb}
  Let $X$ be a $p$-convex Banach lattice and $T\colon E\to X$
  an operator. There is a unique lattice homomorphism
  $\widehat{T}\colon \fbl_p[E]\to X$ such that
  $\widehat{T}\circ\phi_E=T,$ and
  $\norm{\widehat{T}}\le M^{(p)}(X)\,\norm{T},$ where $M^{(p)}(X)$ denotes
  the $p$-convexity constant of $X$.
\end{thm}

\begin{proof}
  As in the proof of \cite[Theorem~2.5]{ART}, there is a unique
  lattice homo\-mor\-phism \mbox{$\widehat{T}\colon L\to X$} such that
  $\widehat{T}(\delta_x) = Tx$ for every $x\in E$.  Our objective is
  to show that
  \begin{equation}\label{t:fblb:Eq2}
    \norm{\widehat T f}_{X}\le M^{(p)}(X)\,\norm{T}\,\norm{f}_p
  \end{equation}
  for every $f\in L$, as this will ensure that~$\widehat{T}$ extends
  uniquely to a lattice homomorphism defined on all of~$\fbl_p[E]$,
  and the extension has norm at most~$M^{(p)}(X)\,\norm{T}$.

  We split the proof of the inequality~\eqref{t:fblb:Eq2} in two
  parts: First we establish it in the special case where $X=L_p(\mu)$
  for some measure space $(\Omega,\Sigma,\mu)$, and then we show how
  to deduce the general version from the special case.

  Thus, suppose first that $X=L_p(\mu)$ for some measure space
  $(\Omega,\Sigma,\mu)$, and let $f\in L$.
  By~\eqref{eq:ABsublattice}, we can write
  $f=\bigvee_{i=1}^n \delta_{x_i}-\bigvee_{j=1}^n \delta_{y_j}$ for
  some $n\in\mathbb N$ and $(x_i)_{i=1}^n$, $(y_j)_{j=1}^n$ in $E$.
  Consider the family of sets $(A_{ij})_{i,j=1}^n\subset \Sigma$
  defined by
  \begin{displaymath}
    A_{ij}
    =\biggl\{\omega\in\Omega\mid\bigvee_{k=1}^n Tx_k(\omega)=Tx_i(\omega),\
    \bigvee_{l=1}^n Ty_l(\omega)=Ty_j(\omega)\biggr\}.
  \end{displaymath}
  Clearly $\bigcup_{i,j=1}^n A_{ij}=\Omega$.  By a standard
  disjointification process, replacing $A_{ij}$ with
  $A_{ij}\setminus\bigcup_{(k,l)\prec(i,j)} A_{kl}$, where $\prec$ is
  any total order on the index set
  $\bigl\{(i,j)\mid 1\le i,j\le n\bigr\}$, we may arrange that the sets
  $(A_{ij})_{i,j=1}^n$ are pairwise disjoint.

  For every $1\le i,j\le n$, define
  \begin{displaymath}
    A_{ij}^+=\bigl\{\omega\in A_{ij}\mid T(x_i-y_j)(\omega) \ge 0\bigr\}
    \qquad\text{and}\qquad
    A_{ij}^-=A_{ij}\setminus A_{ij}^+,
  \end{displaymath}
  and choose positive functions
  $g_{ij},h_{ij}\in L_{p^*}(\mu)=L_p(\mu)^*$, where
  $p^*\in(1,\infty)$ is the conjugate exponent of~$p$, such
  that
  $\norm{g_{ij}}_{L_{p^*}} = \norm{h_{ij}}_{L_{p^*}} =1$,
  \begin{align*}
    \Bignorm{T(x_i-y_j)\chi_{A_{ij}^+}}_{L_p}
    &=\bigl\langle T(x_i-y_j)\chi_{A_{ij}^+}, g_{ij}\bigr\rangle 
     =\int_{A_{ij}^+}T(x_i-y_j)g_{ij}\,d\mu,\\
  \intertext{and}
    \Bignorm{T(y_j-x_i)\chi_{A_{ij}^-}}_{L_p}
    &=\bigl\langle T(y_j-x_i)\chi_{A_{ij}^-}, h_{ij}\bigr\rangle
      =\int_{A_{ij}^-}T(y_j-x_i)h_{ij}\,d\mu.
  \end{align*}
  We may without loss of generality assume that~$g_{ij}$ and~$h_{ij}$
  are supported in $A_{ij}^+$ and~$A_{ij}^-$, respectively. Then the
  set $\bigl\{g_{ij},h_{ij}\mid 1\le i,j\le n\bigr\}$ is
  $1$-equivalent to the unit vector basis of~$\ell_{p^*}^{2n^2}$, and
  consequently the functionals $x_{ij}^*=T^*g_{ij}\in E^*$ and
  $y_{ij}^*=T^*h_{ij}\in E^*$ satisfy
  \begin{multline*}
     \biggl(\sum_{i,j=1}^n\bigabs{x_{ij}^*(x)}^p+\bigabs{y_{ij}^*(x)}^p\biggr)^{\frac1p}
     =\biggl(\sum_{i,j=1}^n \biggabs{\int_{A_{ij}^+}(Tx)g_{ij}\,d\mu}^p
     +\biggabs{\int_{A_{ij}^-}(Tx)h_{ij}\,d\mu}^p\biggr)^{\frac1p}\\
     =\sup\biggl\{\sum_{i,j=1}^n a_{ij}\int_{A_{ij}^+}(Tx)g_{ij}\,d\mu
     +b_{ij}\int_{A_{ij}^-}(Tx)h_{ij}\,d\mu
     \mid  (a_{ij},b_{ij})_{i,j=1}^n\in B_{\ell_{p^*}^{2n^2}} \biggr\}\\
     \le \sup\bigl\{\langle Tx, g\rangle \mid  g\in B_{L_{p*}(\mu)}\bigr\}
     =\norm{Tx}_{L_p}
  \end{multline*}
  for every $x\in E$. Taking the supremum over $x\in B_E$, we conclude that
  \begin{equation}\label{t:fblb:Eq1}
    \weaksumnorm{p}{x^*_{11},\ldots,x^*_{nn},y^*_{11},\ldots,y^*_{nn}}
    \le\norm{T}.
  \end{equation}
  Since $g_{ij}$ is positive, the definition of $A^+_{ij}$ yields that
  \begin{multline*}
     \bigabs{f(x_{ij}^*)}
     =\biggabs{\bigvee_{k=1}^n\int(Tx_k)g_{ij}\,d\mu
       -\bigvee_{l=1}^n\int(Ty_l)g_{ij}\,d\mu}\\
     =\int T(x_i-y_j)g_{ij}\,d\mu=\bignorm{T(x_i-y_j)\chi_{A_{ij}^+}}_{L_p}
     =\bignorm{(\widehat{T} f )\chi_{A_{ij}^+}}_{L_p}
  \end{multline*}
  and similarly
  $\bigabs{f(y_{ij}^*)}=\bignorm{(\widehat{T} f )\chi_{A_{ij}^-}}_{L_p}$ for every
  $1\le i,j\le n$.  Combining~\eqref{Eq:FBLpNorm}
  and~\eqref{t:fblb:Eq1} with these identities, we deduce that
  \begin{multline*}
    \norm{T}\,\norm{f}_p
    \ge\biggl(\sum_{i,j=1}^n \bigabs{f(x_{ij}^*)}^p
         +\bigabs{f(y_{ij}^*)}^p\biggr)^{\frac1p}\\
    =\biggl(\sum_{i,j=1}^n\bignorm{(\widehat{T} f) \chi_{A_{ij}^+}}^p_{L_p}
         +\bignorm{(\widehat{T} f) \chi_{A_{ij}^-}}^p_{L_p}\biggr)^{\frac1p}
    =\norm{\widehat{T} f}_{L_p},
  \end{multline*}
  which establishes~\eqref{t:fblb:Eq2} for $X=L_p(\mu)$ because
  $M^{(p)}(L_p(\mu))=1$.

  \smallskip

  We are now ready to tackle the general case where~$X$ is an
  arbitrary $p$-convex Banach lattice. Given $f\in L$, choose
  $x^*\in X^*_+$ with $\norm{x^*}=1$ and
  $x^*\bigl(\abs{\widehat{T}f}\bigr)=\norm{\widehat{T}f}_X$. Let
  $N_{x^*}$ denote the null ideal generated by~$x^*$, that is,
  $N_{x^*}=\bigl\{x\in X\mid x^*\bigl(\abs{x}\bigr)=0\bigr\}$, and
  let~$Y$ be the completion of the quotient lattice~$X/N_{x^*}$ with
  respect to the norm
  $\norm{x+N_{x^*}}:=x^*\bigl(\abs{x}\bigr)$. Since this is an
  abstract $L_1$-norm, $Y$ is lattice isometric
  to~$L_1(\Omega,\Sigma,\mu)$ for some measure space
  $(\Omega,\Sigma,\mu)$ (see, e.g., \cite[Theorem~1.b.2]{LT2}). The
  canonical quotient map of $X$ onto $X/N_{x^*}$ induces a lattice
  homomorphism $Q\colon X\rightarrow L_1(\Omega,\Sigma,\mu)$ with
  $\norm{Q}=1$. For our purposes, we may without loss of generality
  assume that $(\Omega,\Sigma,\mu)$ is $\sigma$-finite, passing for
  instance to the band generated by $Q(\widehat{T} f)$.

  Since $Q$ is a lattice homomorphism and $X$ is $p$-convex,  we have
  \begin{displaymath}
    \biggnorm{\Bigl(\sum_{k=1}^n\bigabs{Q(x_k)}^p\Bigr)^{\frac{1}{p}}}_{L_1(\mu)}
    \le\biggnorm{\Bigl(\sum_{k=1}^n\abs{x_k}^p\Bigr)^{\frac{1}{p}}}_X
    \le M^{(p)}(X)\,\Bigl(\sum_{k=1}^n\norm{x_k}_X^p\Bigr)^{\frac{1}{p}}
  \end{displaymath}
  for every $n\in\mathbb N$ and $x_1,\dots, x_n\in X$.  Hence the
  Maurey--Nikishin Factorization Theorem (see, e.g,
  \cite[Theorem~7.1.2.]{alb-kal}, and recall that $p<\infty$) yields a
  positive function $h\in L_1(\Omega,\Sigma,\mu)$ with
  $\int_\Omega h\,d\mu=1$ such that $Q$ is bounded if we regard it as
  an operator into~$L_p(h\,d\mu)$. More precisely, we have a
  factorization diagram
  \begin{displaymath}
    \xymatrix{X\ar[d]_S\ar[rr]^{Q}&&L_1(\mu)\\
     L_p(h\,d\mu)\ar@{^{(}->}[rr]&& L_1(h\,d\mu),\ar_{j_h}[u]}
  \end{displaymath}
  where $Sx=h^{-1}Qx$ satisfies $\norm{S}\le M^{(p)}(X)$ and $j_h(g)=gh$
  is an isometric embedding. Note in particular that $S$ is also a
  lattice homomorphism.

  Let us now consider the composite operator
  $R=S\circ T\colon E\to L_p(h\,d\mu)$. By the first part of the
  proof, we know that there is a unique lattice homomorphism
  $\widehat{R}\colon\fbl_p[E]\to L_p(h\,d\mu)$ such that
  $\widehat{R}(\delta_x) = Rx$ for every $x\in E$, and
  $\norm{\widehat{R}}=\norm{R}\le M^{(p)}(X)\,\norm{T}$. Since
  $S\circ \widehat{T}$ and $\widehat{R}$ are lattice homomorphisms
  which agree on the set $\{\delta_x\mid x\in E\}$, it follows that
  $S\circ \widehat{T}=\widehat{R}|_L$. Hence we have
  \begin{multline*}
    \norm{\widehat{T} f}_X
    =x^*\bigl(\abs{\widehat{T} f}\bigr)
    =\bignorm{Q(\widehat{T}f)}_{L_1(\mu)}
    \le\bignorm{S(\widehat{T}f)}_{L_p(hd\mu)}\\
    =\norm{\widehat{R} f}_{L_p(h\,d\mu)}
    \le M^{(p)}(X)\,\norm{T}\,\norm{f}_p.\qedhere
  \end{multline*}
\end{proof}

\begin{rem}
  We do not know of an explicit formula for the norm of the free
  Banach lattice with upper $p$\nobreakdash-esti\-mates, even if the
  constant is one. In fact, we do not know whether this space --- or
  $\fbl^\mathcal{D}[E]$ in general --- can be realized as a lattice of
  functions on the dual ball of~$E^*$. Formulating the latter question
  more rigorously, we ask: Is it true that lattice homomorphisms from
  $\fbl^\mathcal{D}[E]$ to $\mathbb{R}$ separate the points of
  $\fbl^\mathcal{D}[E]$?
\end{rem}

\section{Nonlinear $(p,q)$-summing maps and applications}
\label{section:pqsumming}

The purpose of this section is to explore the norm~\eqref{Eq:FBLpNorm}
further.  It has an obvious similarity with the $p$-summing norm of a
linear operator. A very substantial body of literature is devoted to
the study of $p$-summing norms, their applications, and
generalizations in the linear case. We refer to \cite{DJT} for a
comprehensive exposition of this theory.  Our aim is to establish
analogues of a few of these classical results in our setting.

We begin by introducing a more general version of the
norm~\eqref{Eq:FBLpNorm} involving two indices $1\le p,q < \infty$
and investigating the fundamental properties of this new norm.  For a
Banach space~$E$ and a function $f\in H[E]$, define
\begin{equation}\label{Defn:pqsumnorm}
  \norm{f}_{p,q}
  =\sup\biggl\{\Bigl(\sum_{k=1}^n \bigabs{f(x_k^*)}^p\Bigr)^{\frac{1}{p}}
  \mid n\in\mathbb N,\, x_1^*,\dots,x_n^*\in E^*,\,
  \weaksumnorm{q}{x_1^*,\ldots,x_n^*}\le 1\biggr\}
\end{equation}
and 
\begin{displaymath}
  H_{p,q}[E] = \bigl\{ f\in H[E] \mid  \norm{f}_{p,q} < \infty\bigr\}.
\end{displaymath}
Denote by $\norm{\,\cdot\,}_\infty$ the supremum norm on $B_{E^*}$,
and let $H_\infty [E]$ be the sublattice of~$H[E]$ of all positively
homogeneous functions which are bounded on~$B_{E^*}$. Note that
$\norm{f}_\infty\le\norm{f}_{p,q}$ for every $1 \le p,q < \infty$ and
$f\in H[E]$, and consequently $H_{p,q} [E] \subseteq H_\infty [E]$.
Note also that in the notation previously introduced, we have
$\norm{f}_p=\norm{f}_{p,p}$ and $H_{p}[E]=H_{p,p}[E]$.

The following lemma is straightforward.
\begin{lem} 
  Given $1\le p,q <\infty$ and a Banach space $E$, the space
  $\bigl(H_{p,q}[E],\norm{\,\cdot\,}_{p,q}\bigr)$ equipped with the pointwise
  vector lattice operations is a $p$-convex Banach lattice with
  $p$-convexity constant one.
\end{lem}

It is also easy to see that this space is of interest only for $p\ge q$.

\begin{lem}
  Let $1\le p<q<\infty$. Then $H_{p,q}[E]=\{0\}$ for every Banach space~$E$.
\end{lem}

\begin{proof}
  Let $p,q\in[1,\infty)$, and suppose that~$H_{p,q}[E]$ contains a
  non\-zero function~$f$.  Choose $x^* \in E^*$ such that
  $f(x^*)\neq0$. Then, for every $n\in\mathbb N$, we have
  \begin{displaymath}
    n^{\frac{1}{p}}\bigabs{f(x^*)}
    =\Bigl(\sum_{k=1}^n\bigabs{f(x^*)}^p\Bigr)^\frac{1}{p}
    \le\norm{f}_{p,q}\cdot \weaksumnorm{q}{\,\underbrace{x^*,\ldots,x^*}_n\,}
    =\norm{f}_{p,q}\,n^{\frac{1}{q}}\,\norm{x^*},     
  \end{displaymath}
  which implies that 
  \begin{displaymath}
    n^{\frac{1}{p}-\frac{1}{q}}
    \le \frac{\norm{x^*}}{\bigabs{f(x^*)}}\norm{f}_{p,q}.     
  \end{displaymath}
  Since the right-hand side is independent of~$n$, we conclude that
  $\frac{1}{p}-\frac{1}{q}\le 0$, that is, $p\ge q$.
\end{proof} 

Our next result provides the general comparison among these norms. The
argument follows the same approach as in the Inclusion Lemma
\cite[2.8]{DJT}.

\begin{prop}\label{prop:inclusion}
  Let $1 \le q_j \le p_j < \infty$ for $j=1,2$, and suppose that
  $p_1 \le p_2,$ $q_1 \le q_2,$ and
  $\frac{1}{q_1} - \frac{1}{p_1} \le \frac{1}{q_2} -
  \frac{1}{p_2}$. Then
  \begin{displaymath}
    \norm{f}_{p_2, q_2} \le \norm{f}_{p_1 , q_1}
  \end{displaymath}
  for every $f \in H[E]$.  In particular,
  $H_{p_1,q_1} [E] \subseteq H_{p_2 , q_2}[E]$.
\end{prop}

\begin{proof}
  We begin by observing that the result follows easily for
  $q_1 = q_2$, and if $p_1=p_2$, then the inequalities $q_1 \le q_2$
  and
  $\frac{1}{q_1} - \frac{1}{p_1} \le \frac{1}{q_2} - \frac{1}{p_2}$
  imply that $q_1 = q_2$.  Thus, we may assume that $p_1< p_2$ and
  $q_1< q_2$, and then define
  \begin{displaymath}
    \frac{1}{p} = \frac{1}{p_1} - \frac{1}{p_2} ,
    \quad \frac{1}{q} = \frac{1}{q_1} - \frac{1}{q_2},
  \end{displaymath}
  which satisfy $1 < p\le q < \infty$ by the hypotheses. 

  Let $f\in H [E]$ and fix any $x_1^* , \dots , x_n^* \in E^*$ with
  $\weaksumnorm{q_2}{x_1^*,\ldots,x_n^*}\le 1$. For $1\le k \leq n$,
  define $\lambda_k= \bigabs{f(x_k^*)}^{p_2 / p}$. By the homogeneity
  of $f$, we have
  \begin{equation}\label{prop:inclusionEq1}
    \sum_{k=1}^n\bigabs{f(x_k^*)}^{p_2}
    =\sum_{k=1}^n\bigabs{f(\lambda_k x_k^*)}^{p_1}
    \le\norm{f}_{p_1,q_1}^{p_1}\,\weaksumnorm{q_1}
      {\lambda_1x_1^*,\ldots,\lambda_nx_n^*}^{p_1}. 
  \end{equation}
  H\"{o}lder's inequality shows that
  \begin{displaymath} 
    \Bigl(\sum_{k=1}^n\bigabs{\lambda_k x_k^* (x)}^{q_1}\Bigr)^\frac{1}{q_1}
    \le\Bigl(\sum_{k=1}^n \lambda_k^q\Bigr)^\frac{1}{q}
    \Bigl(\sum_{k=1}^n\bigabs{x_k^* (x)}^{q_2}\Bigr)^\frac{1}{q_2}
    \le\Bigl(\sum_{k=1}^n \lambda_k^q\Bigr)^\frac{1}{q}
    \le\Bigl(\sum_{k=1}^n \lambda_k^p\Bigr)^\frac{1}{p}
  \end{displaymath}
  for every $x\in B_E$ because $\weaksumnorm{q_2}{x_1^*,\ldots,x_n^*}\le 1$ and $p\le q$.
  Taking the supremum over $x\in B_E$ and using the definition of~$\lambda_k$, we obtain
  \begin{equation}\label{prop:inclusionEq2}
    \weaksumnorm{q_1}{\lambda_1x_1^*,\ldots,\lambda_nx_n^*}
    \le\Bigl(\sum_{k=1}^n \bigabs{f(x_k^*)}^{p_2}\Bigr)^{\frac{1}{p}}.
  \end{equation}
  We now substitute~\eqref{prop:inclusionEq2}
  into~\eqref{prop:inclusionEq1} and rearrange the inequality to
  conclude that
  \begin{displaymath}
    \Bigl(\sum_{k=1}^n \bigabs{f(x_k^*)}^{p_2}\Bigr)^{1-\frac{p_1}{p}}
    \le\norm{f}_{p_1 , q_1}^{p_1}.    
  \end{displaymath}
  This completes the proof because $1-\frac{p_1}{p}=\frac{p_1}{p_2}$.
\end{proof}

\begin{prop}\label{prop:equivnormainfinitoV2}
  Let $E$ be a Banach space whose dual has finite cotype~$r\ge 2,$ and
  suppose that $1 \le q< p <\infty$ satisfy
  $\frac{1}{q} - \frac{1}{p} \ge 1-\frac{1}{r}$. Then
  $H_{p,q}[E] = H_\infty [E]$ with equivalence of norms.
\end{prop}

\begin{proof}
  By \cite[Corollary 11.17]{DJT}, every weakly summable sequence
  in~$E^*$ is strongly $r$-summable, and there exists a constant
  $K > 0$ such that
  \begin{displaymath}
    \Bigl(\sum_{k=1}^n \norm{x_k^*}^{r}\Bigr)^\frac{1}{r}
    \le K\weaksumnorm{1}{x_1^*,\ldots,x_n^*}
  \end{displaymath}
  for every finite sequence $(x_k^*)_{k=1}^n$ in~$E^*$.  Hence, for
  $f\in H_\infty[E]$, we have
  \begin{displaymath}
    \Bigl(\sum_{k=1}^n\bigabs{f(x_k^*)}^{r}\Bigr)^\frac{1}{r}
    \le\norm{f}_\infty\Bigl(\sum_{k=1}^n \norm{x_k^*}^{r}\Bigr)^\frac{1}{r}
    \le K\, \norm{f}_\infty\,\weaksumnorm{1}{x_1^*,\ldots,x_n^*}.
  \end{displaymath}
  Taking the supremum over all $n\in\mathbb N$ and
  $x_1^*,\dots,x_n^*\in E^*$ with
  $\weaksumnorm{1}{x_1^*,\ldots,x_n^*}\leq 1$, we conclude that
  $\norm{f}_{r,1}\le K\norm{f}_\infty$.

  Since $1 \le q<p < \infty$ satisfy
  $\frac{1}{q} - \frac{1}{p}\ge 1 - \frac{1}{r}$, we can apply
  \Cref{prop:inclusion} with $p_2=p$, $q_2=q$, $p_1=r$, and $q_1=1$ to
  obtain $\norm{f}_{p,q}\le\norm{f}_{r,1}\le K\norm{f}_\infty$, so
  that $f\in H_{p,q}[E]$ and the $(p,q)$- and supremum norms are
  equivalent.
\end{proof}

As in the classical setting, the Dvoretzky--Rogers Theorem can be used to show that in general these norms are different:

\begin{prop}\label{prop:pqnoesinfinito} 
  Let $E$ be an infinite-dimensional Banach space, and suppose that
  $1 \le q \le p <\infty$ satisfy
  $\frac{1}{q} - \frac{1}{p} < \frac{1}{2}$. Then
  $H_{p,q} [E] \subsetneq H_\infty [E]$.
\end{prop}

\begin{proof}
  By the Dvoretzky--Rogers Theorem \cite[Theorem 10.5]{DJT}, there
  exists a weakly $q$-summable sequence $(x_k^*)_{k\in\mathbb N}$ in
  $E^*$ which fails to be strongly $p$-summable.  Now consider the
  function $f\colon E^*\to\mathbb R$ defined via $f(x^*)=\norm{x^*}$.
  Clearly, $f\in H_\infty [E]$, and for every $n\in \mathbb N$, we
  have
  \begin{displaymath}
    \Bigl(\sum_{k=1}^n \norm{x_k^*}^p\Bigr)^\frac{1}{p}
    \le\norm{f}_{p,q}\, \weaksumnorm{q}{x_1^*,\ldots,x_n^*}.
  \end{displaymath}
  Letting $n\to\infty$, we see that $\norm{f}_{p,q} = \infty$. Thus
  $f \not \in H_{p,q}[E]$.
\end{proof}

Pietsch's Domination Theorem (see, e.g., \cite[Theorem~2.12]{DJT}) is
a cornerstone of the linear theory of $p$\nobreakdash-summing
operators with several important factorization results among its
consequences. We conclude by providing analogues of
\cite[Propo\-si\-tions~2.12 and~2.13]{ART} for $1\le p <\infty$.

Given a Banach space $E$, equip the unit ball~$B_{E^{**}}$ of its
bidual with the relative weak$^*$ topology, and denote the set of
regular Borel probability measures on~$B_{E^{**}}$ by
$\mathfrak P(B_{E^{**}})$.  This is a convex, weak$^*$ compact subset
of the dual space of~$C(B_{E^{**}})$.  Each measure
$\mu\in \mathfrak P(B_{E^{**}})$ induces a function
$f^p_\mu\colon E^*\to \mathbb R_+$ via the definition
\begin{displaymath}
  f_\mu^p (x^*)=\biggl(\int_{B_{E^{**}}}\bigabs{x^*(x^{**})}^p \,
  d\mu(x^{**})\biggr)^\frac{1}{p}
\end{displaymath}
for every $x^*\in E^*$.  This provides a link between $H_p[E]_+$ and
$\mathfrak P(B_{E^{**}})$, as we now explain.

\begin{prop}\label{prop:fmu}
  Let $1\le p <\infty$ and $\mu\in \mathfrak P(B_{E^{**}})$. Then
  $f_\mu^p \in H_p[E]_+$ with $\norm{f_\mu^p}_p\le 1$.
\end{prop}

\begin{proof} 
  The function $f_\mu^p$ is clearly positive and positively
  homogeneous. For $n\in\mathbb N$ and $x_1^*,\ldots, x_n^*\in E^*$,
  we have
  \begin{multline*}
    \biggl(\sum_{k=1}^n\bigabs{f_\mu^p (x^*_k)}^p\biggr)^\frac{1}{p}
    =\biggl(\int_{B_{E^{**}}}\sum_{k=1}^n\bigabs{x_k^*(x^{**})}^p\,
       d\mu(x^{**})\biggr)^\frac{1}{p}\\
    \le \sup_{x^{**}\in B_{E^{**}}}
       \biggl(\sum_{k=1}^n\bigabs{x_k^*(x^{**})}^p\biggr)^\frac{1}{p}
    =\weaksumnorm{p}{x_1^*,\ldots,x_n^*},
  \end{multline*}
  where the last equality follows from the weak$^*$ density of $B_E$
  in~$B_{E^{**}}$.  Hence $\norm{f_\mu^p}_p\le 1$.
\end{proof}

\begin{prop}\label{prop:Pietsch}
  Let $1\le p <\infty$. For every $f\in H_p[E]_+$, there is a measure
  $\mu\in\mathfrak P(B_{E^{**}})$ such that
  \begin{math}
    f(x^*)\le\norm{f}_p \, f_\mu^p(x^*)     
  \end{math}
  for every $x^*\in E^*$.
\end{prop}

\begin{proof} 
  This proof is based on the proof of Pietsch's Domination Theorem
  given in \cite[2.12]{DJT}. For every non\-empty finite subset~$M$ of
  $E^{*}$, define $g_M \colon B_{E^{**}}\to\mathbb R$ by
  \begin{displaymath}
    g_M(x^{**})
    =\sum_{x^* \in M}\Bigl(f(x^*)^p -\norm{f}_p^p \cdot\bigabs{x^*(x^{**})}^p\Bigr).
  \end{displaymath}
  Then $g_M$ is weak$^*$ continuous, and so the set~$Q$ of all such
  functions $g_M$ is contained in~$C(B_{E^{**}})$. Given non\-empty
  finite subsets $M_1$ and $M_2$ of $E^{*}$ and $0<\lambda<1$, the
  positive homogeneity of $f$ implies that
  $\lambda\cdot g_{M_1}+(1-\lambda)\cdot g_{M_2}=g_{M_3}$, where
  \begin{displaymath}
    M_3= \bigl\{\lambda^{1/p}x^* \mid x^* \in M_1 \bigr\}
    \cup \bigl\{(1-\lambda)^{1/p}x^* \mid x^* \in M_2\bigr\}.    
  \end{displaymath}
  This shows that~$Q$ is a convex set.

  The definition~\eqref{Eq:FBLpNorm} of the norm $\norm{\,\cdot\,}_p$ implies
  that $Q$ is disjoint from the strictly positive cone
  \begin{displaymath}
    P=\Bigl\{h\in C(B_{E^{**}})\mid h(x^{**})>0\
    \text{for every}\ x^{**}\in B_{E^{**}}\Bigr\}.    
  \end{displaymath}
  Since $P$ is open and convex, the geometric version of the
  Hahn--Banach Theorem guarantees the existence of a functional
  $\mu\in C(B_{E^{**}})^{*}$ and a constant $c\in\mathbb R$ such that
  $\mu(g)\leq c<\mu(h)$ for every $g\in Q$ and $h\in P$.

  Choosing $M=\{0\}\subseteq E^*$, we have $g_M=0$. Therefore
  $0\in Q$, and so $c\ge 0$. On the other hand, as every strictly
  positive constant function belongs to $P$, we must have $c\le
  0$. It follows that $c=0$, which implies that $\mu(h)\ge 0$ for
  every $h\in C(B_{E^{**}})_{+}$. In other words, $\mu$ is a positive
  regular Borel measure such that
  \begin{displaymath}
    \int_{B_{E^{**}}} g\,d\mu\le 0<\int_{B_{E^{**}}} h\, d\mu
  \end{displaymath}
  for every $g\in Q$ and $h\in P$.  This inequality is unaffected by
  scaling of $\mu$, so we may assume that
  $\mu\in\mathfrak{P}(B_{E^{**}})$. For every $x^* \in E^*$, the
  function $g_{\{x^*\}}$ belongs to~$Q$, and therefore
  \begin{displaymath}
    0\ge \int_{B_{E^{**}}}\Bigl(f(x^*)^p
    -\norm{f}_p^p \cdot\bigabs{x^*(x^{**})}^p\Bigr) d\mu(x^{**})
    =f(x^*)^p - \norm{f}_p^p\, f_\mu^p(x^*)^p 
  \end{displaymath}
  because $\mu$ is a probability measure.
\end{proof}

We can summarize the conclusions of Propositions~\ref{prop:fmu}
and~\ref{prop:Pietsch} as follows.

\begin{cor}
  Let $1\le p <\infty$ and $f\in H[E]_+$. Then $f\in H_p[E]_+$ if and
  only if, for some constant $C>0,$ there is a measure
  $\mu\in\mathfrak{P}(B_{E^{**}})$ such that
  \begin{math}
    f(x^*)\le C\cdot f_\mu^p(x^*) 
  \end{math}
  for every $x^* \in E^*$. Furthermore, when $f\in H_p[E]$, its norm
  $\norm{f}_p$ can be computed as the infimum of all constants~$C$ for
  which such a measure~$\mu$ exists.
\end{cor}

\section*{Acknowledgements}
P.~Tradacete gratefully acknowledges support by Agencia Estatal de Investigaci\'on (AEI) and Fondo Europeo de Desarrollo Regional (FEDER) through grants MTM2016-76808-P (AEI/FEDER, UE) and MTM2016-75196-P (AEI/FEDER, UE), as well as Spanish Ministry of Science and Innovation, through the ``Severo Ochoa Programme for Centres of Excellence in R\&D'' (CEX2019-000904-S) and from Consejo Superior de Investigaciones Cient\'ificas (CSIC), through  ``Ayuda extraordinaria a Centros de Excelencia Severo Ochoa'' (20205CEX001).

H.~Jard\'on-S\'anchez gratefully acknowledges support by Fundaci\'on Mar\'ia Cristina Masaveu Peterson through its Scholarships for Academic Excellence and Fundaci\'on ``la Caixa'' through its Postgraduate Studies in Europe Fellowships.

V.G.~Troitsky gratefully acknowledges support by Natural Sciences and
Engineering Research Council of Canada.

\end{document}